\documentclass{amsart}
\usepackage{amssymb,
mathrsfs,
hyperref,
tikz,
euscript,
dsfont,
amsmath,
upgreek,
enumitem
}
\hypersetup{colorlinks=true}
\numberwithin{equation}{section}
\usepackage[T1]{fontenc}
\usepackage[shadow]{todonotes}

\def\N{\EuScript N}
\newcommand{\M}{\EuScript{M}}
\def\A{\EuScript A}
\def\B{\EuScript B}

\def\L{\EuScript L}
\def\Q{\EuScript Q}
\def\R{\EuScript R}

\def\EF{\mathrm{EF}}
\def\imp{\rightarrow}

\def\IMP{\Rightarrow}
\def\IFF{\Leftrightarrow}

\def\wins{\uparrow}
\newcommand{\1}{\mathbf{I}}
\newcommand{\2}{\mathbf{II}}

\renewcommand{\L}{\mathcal{L}}
\newcommand{\RR}{\mathbb{R}}

\newcommand{\pre}[2]{{}^{#1}#2}

\newcommand{\On}{\mathrm{On}}

\newcommand{\Nbhd}{\boldsymbol{N}}
\newcommand{\Tr}{\mathscr{T}}
\newcommand{\ElCl}{\mathcal{C}}

\newcommand{\ran}{\operatorname{ran}}
\newcommand{\Mod}{\operatorname{Mod}}

\newcommand{\pred}{\operatorname{pred}}
\newcommand{\dom}{\operatorname{dom}}

\newenvironment{enumerate-(a)}{\begin{enumerate}[label={\upshape (\alph*)}, leftmargin=2pc]}{\end{enumerate}}
\newenvironment{enumerate-(a)-r}{\begin{enumerate}[label={\upshape (\alph*)}, leftmargin=2pc,resume]}{\end{enumerate}}
\newenvironment{enumerate-(a)-5}{\begin{enumerate}[label={\upshape (\alph*)}, leftmargin=2pc,start=5]}{\end{enumerate}}
\newenvironment{enumerate-(A)}{\begin{enumerate}[label={\upshape (\Alph*)}, leftmargin=2pc]}{\end{enumerate}}
\newenvironment{enumerate-(A)-r}{\begin{enumerate}[label={\upshape (\Alph*)}, leftmargin=2pc,resume]}{\end{enumerate}}
\newenvironment{enumerate-(i)}{\begin{enumerate}[label={\upshape (\roman*)}, leftmargin=2pc]}{\end{enumerate}}
\newenvironment{enumerate-(i)-r}{\begin{enumerate}[label={\upshape (\roman*)}, leftmargin=2pc,resume]}{\end{enumerate}}
\newenvironment{enumerate-(I)}{\begin{enumerate}[label={\upshape (\Roman*)}, leftmargin=2pc]}{\end{enumerate}}
\newenvironment{enumerate-(I)-r}{\begin{enumerate}[label={\upshape (\Roman*)}, leftmargin=2pc,resume]}{\end{enumerate}}
\newenvironment{enumerate-(1)}{\begin{enumerate}[label={\upshape (\arabic*)}, leftmargin=2pc]}{\end{enumerate}}
\newenvironment{enumerate-(1)-r}{\begin{enumerate}[label={\upshape (\arabic*)}, leftmargin=2pc,resume]}{\end{enumerate}}
\newenvironment{itemizenew}{\begin{itemize}[leftmargin=2pc]}{\end{itemize}}

\newtheorem{theorem}{Theorem}[section]
\newtheorem{lemma}[theorem]{Lemma}
\newtheorem{cor}[theorem]{Corollary}
\newtheorem{prop}[theorem]{Proposition}
\newtheorem{question}[theorem]{Question}

\newtheorem{fact}[theorem]{Fact}
\newtheorem{claim}{Claim}[theorem]
\theoremstyle{definition}
\newtheorem{defin}[theorem]{Definition}
\newtheorem{example}[theorem]{Example}

\theoremstyle{remark}
\newtheorem{remark}[theorem]{Remark}

\begin{document}

\title{A descriptive Main Gap Theorem}
\date{\today}
\author[F.~Mangraviti]{Francesco Mangraviti}
\address{Institut f\"ur Philosophie I, Ruhr Universit\"at Bochum,
	Universit\"atsstr.\ 150,
	44801 Bochum --- Germany}
\email{Francesco.Mangraviti@ruhr-uni-bochum.de}
\author[L.~Motto Ros]{Luca Motto Ros}
\address{Dipartimento di matematica \guillemotleft{Giuseppe Peano}\guillemotright, Universit\`a di Torino, Via Carlo Alberto 10, 10123 Torino --- Italy}
\email{luca.mottoros@unito.it}
\thanks{We thank the referee for the careful reading of our manuscript and M.\ Moreno for suggesting and discussing with us the content of Section~\ref{subsec:incompletetheories}.}

\begin{abstract} 
Answering one of the main questions of \cite[Chapter 7]{FHK14}, we show that there is a tight connection between the depth of a classifiable shallow theory \( T \) and the Borel rank of the isomorphism relation \( \cong^\kappa_T \) on its models of size \( \kappa \), for \( \kappa \) any cardinal satisfying \( \kappa^{< \kappa} = \kappa > 2^{\aleph_0} \). This is achieved by establishing a link between said rank and the \( \mathcal{L}_{\infty \kappa} \)-Scott height of the \( \kappa \)-sized models of \( T \), and yields to the following descriptive set-theoretical analogue of Shelah's Main Gap Theorem: Given a countable complete first-order theory \( T \), either \( \cong^\kappa_T \) is Borel with a \emph{countable} Borel rank (i.e.\ very simple, given that the length of the relevant Borel hierarchy is \( \kappa^+ > \aleph_1 \)), or it is not Borel at all. The dividing line between the two situations is the same as in Shelah's theorem, namely that of classifiable shallow theories. We also provide a Borel reducibility version of the above theorem, discuss some limitations to the possible (Borel) complexities of \( \cong^\kappa_T \), and provide a characterization of categoricity of \( T \) in terms of the descriptive set-theoretical complexity of \( \cong^\kappa_T \). 
\end{abstract}

\maketitle

\section{Introduction}

\emph{In the whole paper, (first-order) theories, usually denoted by \( T \), are assumed to be countable, complete and to have infinite models, \textbf{unless otherwise stated}.}

 Classification theory (also known as stability theory) was first conceived as a tool to solve in a systematic and general way the \emph{spectrum problem} for countable complete theories, that is, the problem of computing the number \( I(\kappa,T) \) of nonisomorphic models  of \( T \) of size \( \kappa \geq \aleph_1 \). The obvious bounds for \( I(\kappa,T) \) are
\[ 
1 \leq I(\kappa,T) \leq 2^\kappa.
 \] 

The main idea of classification theory, as shown in Shelah's  masterpiece \cite{Sh00}, is that there are several key dichotomies that can be used to identify how well-behaved a theory is: (super)stable versus un(super)stable, DOP (\emph{Dimensional Order Property}) versus NDOP (\emph{Not-DOP}), OTOP (\emph{Omitting Types Order Property}) versus NOTOP (\emph{Not-OTOP}), shallow versus deep, and so on. 
Shelah first proved that if a theory \( T \) is either unsuperstable, or superstable and either DOP or OTOP, then the spectrum function always assumes the maximal value, i.e.\ \( I(\kappa,T) = 2^\kappa \) for every \( \kappa \geq \aleph_1 \).
Thus theories \( T \) which are (stable) superstable, NDOP and NOTOP are the only ones for which there can be a nontrivial upper bound on the spectrum function, and for this reason such \( T \)'s are called \emph{classifiable}. 

The following quote from ~\cite{Bal88} concisely explains how the spectrum problem for classifiable theories was solved by Shelah:
\begin{quote}
The solution of the spectrum problem for classifiable theories  depends upon a key construction which assigns to each model of
size \( \kappa \) a skeleton of submodels. Each submodel has cardinality at most
\( 2^{\aleph_0} \), and the skeleton is partially ordered by the natural tree order on a
subset of \( \pre{<\omega}{\kappa} \). The isomorphism type of the model is determined by the
small submodels and this partial ordering. [...] If one of
these trees is not well-founded, the theory is said to be \textbf{deep} and has \( 2^\kappa \)
models for every \( \kappa \geq \aleph_1 \). If not, the theory is \textbf{shallow} and the
type of structure theory we have described exists. We are able to assign
to each such shallow theory a depth \(\alpha\) corresponding to the rank of a
system of invariants, as discussed above, and to compute the spectrum
function of \( T \) in terms of that depth.
\end{quote}
\begin{flushright}
(John T.\ Baldwin, \emph{Fundamentals of Stability Theory})
\end{flushright}
The above ``decomposition'' technique yields that an upper bound to the number of  isomorphism types for \( \kappa \)-sized models of a classifiable shallow theory \( T \) can essentially be obtained by computing how many labeled (with labels of size\linebreak \( \leq 2^{\aleph_0} \)) well-founded subtrees of \( \pre{< \omega}{\kappa} \) of rank \( \leq \alpha \) are there. Summing up all these informations, one finally gets Shelah's celebrated  Main Gap Theorem.
\begin{theorem}[{\cite[The Main Gap Theorem 6.1]{Sh00}}] \label{thm:She}
Let \( \kappa \geq \aleph_1 \) be the \( \gamma \)-th cardinal. 
\begin{enumerate-(1)} \label{thm:Shelah}
\item \label{Shelah1}
If \( T \) is classifiable shallow of depth \(\alpha\), then
\[
I(\kappa, T) \leq \beth_\alpha \left(|\gamma|^{2^{\aleph_0}} \right) .
\] 
(If \( \alpha \geq \omega \) then \( \beth_\alpha \left(|\gamma|^{2^{\aleph_0}} \right) \) can be replaced by \( \beth_\alpha \left(|\gamma| \right) \).)
\item
If \( T \) is not classifiable shallow, then
\[
I(\kappa, T) = 2^\kappa.
\]
\end{enumerate-(1)}
\end{theorem}

Since by~\cite[Th\'eor\`eme 4.1]{La85} classifiable shallow theories have countable depth, when we are in case~\ref{Shelah1} of the above theorem we actually get a uniform upper bound on \( I(\kappa,T) \) which is independent of the depth of \( T \), namely
\[
I(\kappa, T) <  \beth_{\omega_1} \left(|\gamma| \right).
\] 
\begin{remark} \label{rmk:She}
The upper bound in Theorem~\ref{thm:Shelah}\ref{Shelah1} may become trivial (e.g.\ when \( \kappa \) is a fixed point of the \( \aleph \)-function),  but it is not when e.g.\ \( \kappa = \aleph_\gamma \) is such that 
\[ 
\beth_{\omega_1}\left( | \gamma| \right) \leq \kappa .
\]
Indeed, in this case Shelah's upper bound is even \( < \kappa \).
In general it is easy to find cardinals satisfying the above condition. For example, under \( \mathsf{GCH} \) there are unboundedly many such \( \kappa \)'s: if \( \gamma , \delta \geq  \omega_1 \) with \( |\gamma| \geq |\delta| \), then every \( \kappa = \aleph_{\gamma+\delta} \) does the job. In particular, letting \(\delta\) vary over all uncountable ordinals we get examples of such \( \kappa \) which are either successors or singular cardinals of any cofinality. 
\end{remark}

The Main Gap Theorem can be taken as evidence that Shelah's notion of a classifiable shallow theory does in fact capture the general idea of model-theoretic ``simplicity''. 
 Such theories appear quite naturally in mathematics: some well-known examples are the theory of algebraically closed fields of fixed characteristic (along with all uncountably categorical theories) and the theory of the additive group of integers.

The reader may wonder why so far we have only considered \emph{uncountable} models. One thing to note is that, in contrast to the uncountable case, we do not yet know how many \emph{countable} models a  theory \( T \) may have in general. Indeed, Vaught's conjecture, asserting that either
\( I(\aleph_0, T) \leq \omega \) or \( I(\aleph_0, T) = 2^{\aleph_0} \), is still one of the major open problems in model theory. 
One of the strategies devised to tackle this problem in the Nineties was that of using methods from (classical) descriptive 
set theory. The starting point of this approach is that countable structures can naturally be coded as elements of the 
Cantor space $\pre{\omega}{2}$ (i.e. countable binary sequences), so that the isomorphism relation \( \cong^\omega_T \) on 
countable models of \( T \) may be construed as an analytic equivalence relation on such space. Some progress has been 
obtained through this method: for example, Silver's theorem~\cite{Sil80} yields that Vaught's conjecture holds for those 
theories \(  T \) for which the isomorphism relation \( \cong^\omega_T \) is Borel. The latter condition may be seen as a 
simplicity notion itself. Indeed, \( \cong^\omega_T \) is Borel if and only if there is an ``effective'' procedure which, using 
only countable set-theoretical operations such as unions, intersections, and complements, allows us to determine whether 
two countable models of \( T \) are isomorphic or not --- in other words, there is a Borel procedure to classify the countable 
models of \( T \) up to isomorphism. Unfortunately, there is no relation between Shelah's classification of \( T \) in terms of its stability properties and the simplicity of \( \cong^\omega_T \) in the descriptive set-theoretic sense: for example, the theory of dense linear orders is unstable, 
but the isomorphism relation on its countable models is very simple (it is a Borel equivalence relation with Borel rank 2 and only \( 4 \) different classes); conversely, in~\cite{MK11} it is shown that there are theories \( T \) which are very simple stability-wise, but such that \( \cong^\omega_T \) is not even Borel.

This failure forces us to move to the uncountable setting again. Replacing \( \omega \) with an uncountable cardinal \( \kappa \), it is easy to check that, up to isomorphism, all \( \kappa \)-sized structures can be coded as elements of the \emph{generalized Cantor space} \( \pre{\kappa}{2} \), i.e.\ the space of all binary \( \kappa \)-sequences equipped with the so-called bounded topology, a natural generalization of the standard topology on \( \pre{\omega}{2} \) (see Section~\ref{sec:genDST}). Despite the fact that \( \pre{\kappa}{2} \) is no longer a Polish space, it is still possible to naturally mirror all classical definitions in the new setting: for example, Borel sets are replaced by \( \kappa^+ \)-Borel ones (i.e.\ by the sets in the smallest \( \kappa^+ \)-algebra generated by the open sets), analytic sets are replaced by \( \kappa \)-analytic ones (i.e.\ continuous images of \( \kappa^+ \)-Borel sets), and so on. Even though the resulting theory, which is nowadays called \emph{generalized descriptive set theory}, presents many differences from the classical theory and is severely affected by a myriad of independence phenomena already for very simple sets, some of the basic features are preserved. For example, in~\cite[Lemma 4.15, and Proposition 4.19]{AMR16} it is shown that the \( \kappa^+ \)-Borel subsets of \( \pre{\kappa}{2} \) can be stratified in a hierarchy with \( \kappa^+ \)-many levels,%
\footnote{If \( \kappa^{< \kappa} \neq \kappa \), the argument to prove this is quite different from the one used in the classical setting \( \kappa = \omega \).}
 so that to each \( \kappa^+ \)-Borel set \( A \subseteq \pre{\kappa}{2} \) we can assign an ordinal \( \mathrm{rk}_B(A) < \kappa^+ \), called \emph{Borel rank}, measuring its complexity. For simplicity of notation, we stipulate that \( \mathrm{rk}_B(A) = \infty \) whenever \( A \) is not \( \kappa^+ \)-Borel.

Working in this new setup, one can show that the set of (codes for) \( \kappa \)-sized models of a given theory \( T \) form a \( \kappa^+ \)-Borel set, and that the isomorphism relation on it, which will be denoted by \( \cong^\kappa_T \), is a \( \kappa \)-analytic equivalence relation. It is thus natural to ask how much ``simple'' \( \cong^\kappa_T \) is from the (generalized) descriptive set-theoretical point of view.

\begin{question} \label{quest:FHK}
For which theories \( T \) and which uncountable cardinals \( \kappa \)  does it happen that \( \cong^\kappa_T \)
is \( \kappa^+ \)-Borel?  What can we say about the Borel rank of \( \cong^\kappa_T \)?
\end{question}

A finer question is

\begin{question} \label{quest:FHK2}
Does the \( \kappa^+ \)-Borelness (and/or the Borel rank) of \( \cong^\kappa_T \) depend on both parameters \( \kappa \) and \(T \), or does it just depend on the theory \( T \)?
\end{question}

In the latter case one could regard the theory \( T \) itself as ``simple'' if some/any of the \( \cong^\kappa_T \)'s is \( \kappa^+ \)-Borel.

The first part of Question~\ref{quest:FHK} and Question~\ref{quest:FHK2} were answered by S.-D.\ Friedman, T.\ Hyttinen, and V.\ Kulikov in their impressive and seminal work~\cite{FHK14}.

\begin{theorem}[S.-D.\ Friedman-Hyttinen-Kulikov, {\cite[Theorem 63]{FHK14}}] \label{thm:FHK}
Let \( \kappa \) be such that%
\footnote{In~\cite[Theorem 63]{FHK14} it is further required that \( \kappa \) is not weakly inaccessible. However, as it can be checked following the proofs below, there is no need to add this restriction to obtain the above result and its refinement.}
 \( \kappa^{< \kappa} = \kappa > 2^{\aleph_0} \).
\begin{enumerate-(1)}
\item
If \( T \) is classifiable shallow, then \( \cong^\kappa_T \) is \( \kappa^+ \)-Borel.
\item
If \( T \) is not classifiable shallow, then \( \cong^\kappa_T \) is \emph{not} \( \kappa^+ \)-Borel.
\end{enumerate-(1)}
\end{theorem}

Remarkably, the dividing line distinguishing whether \( \cong^\kappa_T \) is \( \kappa^+ \)-Borel or not is thus the same as in Shelah's Main Gap Theorem~\ref{thm:She}. The following question, although with a slightly different formulation, may be found as  one of the main open problems listed in~\cite[Chapter 7]{FHK14}.

\begin{question} \label{ques:main}
If \( T \) is classifiable shallow, what is the Borel rank of \( \cong^\kappa_T \)? Is it related to the depth of \( T \)?
\end{question}

The goal of this paper is precisely to address this and other related problems. After proving  in Section~\ref{sec:genDST} some (old and new) preliminary results about generalized descriptive set theory, as a first step we provide in Section~\ref{subsec:categoricity}
a purely descriptive set-theoretical characterizion of \( \kappa \)-categoricity by showing that all \( \kappa \)-sized models of a theory \( T \) are isomorphic (i.e.\ \( T \) is \( \kappa \)-categorical) if and only if \( \cong^\kappa_T \) is (cl)open (Theorem~\ref{thm:catcharbounded}).

In Section~\ref{sec:ScottVsBorel} we carefully analize the Friedman-Hyttinen-Kulikov's proof of Theorem~\ref{thm:FHK} and obtain the following result (see also Theorem~\ref{orig3}) connecting the Borel rank of \( \cong^\kappa_T \) to the \( \mathcal{L}_{\infty \kappa} \)-Scott height of its \( \kappa \)-sized models, which may be of independent interest. Given a theory \( T \) and a cardinal \( \kappa \geq \aleph_1 \), set
\[ 
B(\kappa,T) = \mathrm{rk}_B(\cong^\kappa_T) .
 \] 
Recall that if \( \cong^\kappa_T \) is \( \kappa^+\)-Borel the obvious bounds on \( B(\kappa,T) \) are \( 0 \leq B(\kappa,T) < \kappa^+ \), while if \( \cong^\kappa_T \) is not \( \kappa^+ \)-Borel then \( B(\kappa,T) = \infty \) (which here is considered as the maximal complexity). Let also \( S(\kappa,T) \) be the supremum of the \( \mathcal{L}_{\infty \kappa} \)-Scott heights of the \( \kappa \)-sized models of the theory \( T \) (see Section~\ref{inf}). It can be shown that \( S(\kappa,T) \) is either \( \leq \kappa^+ \) or else undefined, in which case we set \( S(\kappa,T)= \infty \).

\begin{theorem} \label{thm:ranksintro}
Let \( \kappa \) be such that \( \kappa^{< \kappa} = \kappa \).  Then
\begin{itemize}
\item
if \( S(\kappa,T) < \kappa^+ \), then \( B(\kappa,T) \leq 2 S(\kappa,T) < \kappa^+ \);
\item
if \( B(\kappa,T) \neq \infty \), then \( S(\kappa,T) \leq  B(\kappa,T)< \kappa^+ \).
\end{itemize} 
In particular, \( B(\kappa,T) \) and \( S(\kappa,T) \) always have finite distance.
\end{theorem}

This kind of analysis actually applies to a wider setup: indeed, instead of considering just the models of a given first-order theory \( T \), we can pick \emph{any} collection of \( \kappa \)-sized models  \( \ElCl \) closed under isomorphism, and obtain an analogue of Theorem~\ref{thm:ranksintro} for the isomorphism relation \( \cong^\kappa_{\ElCl}\) on \( \ElCl \).
This yields  the following corollary, which generalizes to uncountable \( \kappa \)'s (and to a slightly more general setting) a result obtained in the countable case \( \kappa = \omega \) by Becker and Kechris~\cite[Corollary 7.1.4]{BK96}.

\begin{cor}
Let  \( \kappa \) be such that \( \kappa^{< \kappa} = \kappa \), and let \( \ElCl \) be any collection of \( \kappa \)-sized models closed under isomorphism. Then \( \cong^\kappa_{\ElCl} \) is \( \kappa^+ \)-Borel if and only if there is \( \beta < \kappa^+ \) such that the \( \mathcal{L}_{\infty \kappa} \)-Scott height of any structure in \( \ElCl \) is \( \leq \beta \).
\end{cor}

In Section~\ref{sec:descriptivemaingap}
we use the previous results to solve the ``Borel analogue'' of the spectrum problem, thus sharpening Theorem~\ref{thm:FHK} and answering (at least partially) Question~\ref{ques:main}. 

\begin{theorem}[Descriptive Main Gap Theorem] \label{thm:main}
Let \( \kappa \) be such that \( \kappa^{< \kappa} = \kappa > 2^{\aleph_0} \).
\begin{enumerate-(1)}
\item \label{thm:maincase1}
If \( T \) is classifiable shallow of depth \(\alpha\), then \( B(\kappa,T) \leq  4 \alpha \).
\item
If \( T \) is not classifiable shallow, then \( B(\kappa,T)  = \infty \).
\end{enumerate-(1)}
\end{theorem}

Thus in case~\ref{thm:maincase1}, which corresponds exactly to Theorem~\ref{thm:Shelah}\ref{Shelah1}, the ordinal \( B(\kappa,T) \) is almost everywhere dominated by a \emph{constant} function which, unlike Shelah's upper bound on the number of isomorphism types, depends only on the depth \( \alpha \) of the theory and not on the cardinal \( \kappa \) under consideration. Moreover, in view of the above-mentioned fact that  \( \alpha<\aleph_1 \) by~\cite[Th\'eor\`eme 4.1]{La85},  in case~\ref{thm:maincase1} we can get a nontrivial uniform upper bound which is independent of \(\alpha\) as well, namely
\[ 
B(\kappa,T) < \aleph_1 < \kappa^+.
 \] 
In particular, there is no theory \( T \) with \( \cong^\kappa_T \) of uncountable Borel rank.
Another interesting difference from Shelah's Main Gap is that the upper bound on the Borel rank of \( \cong^\kappa_T \) is almost never trivial for the relevant \( \kappa \)'s: for example, under \( \mathsf{GCH} \) the descriptive gap is non-trivial for every regular cardinal \( \kappa \geq \aleph_2 \) (in particular, for all the successors, with the possible exception of \( \aleph_1 \)).

 Summing up all the mentioned results, we get the  picture described in Table~\ref{tab:sumup} strictly relating the model-theoretic properties of \( T \), the number of its \( \kappa \)-sized models \( I(\kappa,T) \) (up to isomorphism),  the \( \L_{\infty \kappa} \)-Scott height \( S(\kappa,T) \) of \( T \), the topological complexity of \( \cong^\kappa_T \), and its Borel rank \( B(\kappa,T) \).
 
 \begin{table}[h!]
  \begin{center}
    \label{tab:sumup}
    \begin{tabular}{|c|c|c|c|} 
    \hline
    \begin{minipage}[c][1,3cm][c]{2,6cm}\emph{Model-theoretic \\ properties of \( T \)}\end{minipage} & 
    \textbf{\( \kappa \)-categorical} & 
\textbf{Classifiable shallow} & 
    \textbf{Not classifiable shallow} \\
    \hline
     \begin{minipage}[c][1,3cm][c]{2,6cm}\emph{Number of \\ \( \kappa \)-sized models}\end{minipage} & 
    \( I(\kappa,T) = 1 \) \ \ \( ({}^*) \) & 
\( I(\kappa,T) < \beth_{\omega_1}(|\gamma|) \) \ \ \( ({}^\dagger) \)  &
    \( I(\kappa,T) = 2^\kappa \) \ \ \( ({}^\dagger) \) \\
    \hline
     \begin{minipage}[c][1,3cm][c]{2,6cm}\emph{\( \L_{\infty \kappa}\)-Scott \\ height of \( T \)}\end{minipage} &
    \( S(\kappa,T) = 0 \) \ \ \( ({}^* ) \) &
    \( S(\kappa,T) < \aleph_1 \) \ \ \( ({}^\dagger) \) &
\( S(\kappa,T) \in \{ \kappa^+ , \infty \} \) \ \ \( ({}^\dagger) \) \\
    \hline
     \begin{minipage}[c][1,3cm][c]{2,6cm}\emph{Topological \\ complexity of \( \cong^\kappa_T \)}\end{minipage} &
    (Cl)open &
    \( \kappa^+ \)-Borel \ \ \( ({}^\ddagger) \) &
    Not \( \kappa^+ \)-Borel \ \ \( ({}^\ddagger) \) \\
    \hline
     \begin{minipage}[c][1,3cm][c]{2,6cm}\emph{\( \kappa^+ \)-Borel \\ rank of \( \cong^\kappa_T \)}\end{minipage} &
    \( B(\kappa,T) = 0 \) &
    \( B(\kappa,T) < \aleph_1 \) &
    \( B(\kappa,T) = \infty \) \ \ \( ({}^\ddagger) \) \\
    \hline
    \end{tabular}
    
    \bigskip
    
    \caption{Characterizations of some stability notions from model-theory when \( \L \) is a relational language, \( T \) is a countable complete first-order theory, and \( \kappa = \aleph_\gamma \) is an uncountable cardinal such that \( \kappa^{< \kappa} = \kappa > 2^{\aleph_0} \). Entries marked with \( ({}^*) \) are easy reformulations of \( \kappa \)-categoricity, the ones marked with \( ({}^\dagger) \) are due to Shelah~\cite{Sh00}, those marked with \( ({}^\ddagger) \) are due to Friedman, Hyttinen and Kulikov~\cite{FHK14}, while the remaining ones are obtained in this paper.}
  \end{center}
\end{table}

Theorem~\ref{thm:main} imposes \( \aleph_1 \) as an upper bound on the Borel rank of a given 
\( \cong^\kappa_T \), but a full answer to the first part of Question~\ref{ques:main} would require 
to assess which Borel classes with a (necessarily countable) index are actually inhabitated by such an 
isomorphism relation. We address this problem in Sections~\ref{subsec:possiblecomplexities} and~\ref{subsec:someexamples} and provide some partial answers. 
For example, we show that \( \cong^\kappa_T \) can never be a proper 
\( \boldsymbol{\Sigma}^0_\alpha \) set for \( \alpha \) limit (Theorem~\ref{thm:limit}); this is a new (and somehow 
unexpected) observation also in the classical setup of countable models  \( \kappa = \omega \). 
Moreover, when \( \kappa > 2^{\aleph_0} \) satisfies certain additional conditions, then \( \cong^\kappa_T \) 
cannot be a proper \( \boldsymbol{\Sigma}^0_\alpha \) or a proper \( \boldsymbol{\Pi}^0_\alpha \) set 
for \emph{any} ordinal \(\alpha\) (Proposition~\ref{prop:GCH}).  Finally, in Section~\ref{subsec:Borelreducibility} we obtain a variant of the Descriptive Main Gap 
Theorem~\ref{thm:main} in which the complexity of \( \cong^\kappa_T \) is measured using Borel 
reducibility \( \leq^\kappa_B \) rather than  Borel ranks: under certain conditions on \( \kappa \), if \( T \) is classifiable shallow while \( T' \) is not, then \( {\cong^\kappa_T} <_B^\kappa {\cong^\kappa_{T'}} \) and there are equivalence relations lying strictly \( \leq_B^\kappa \)-between the two  (Proposition~\ref{prop:reducibility}).

\section{Generalized descriptive set theory} \label{sec:genDST}

In this section we introduce the tools from generalized descriptive set theory that are used in the sequel. We will prove only those results which are not explicitly proved elsewhere in the literature, referring the reader to~\cite{FHK14,AMR16} for a thorough and detailed exposition of the theory and its basics.

We denote by \( \mathrm{On} \) the class of all ordinal numbers.
Given two sets \( X,Y\) we denote by \( \pre{X}{Y} \) the set of all functions \( f \colon X \to Y \). When \( \alpha \) is an ordinal we set \( \pre{<\alpha}{Y} = \bigcup_{\beta<\alpha} \pre{\beta}{Y} \).
For the rest of this section, let \( \kappa \) be an infinite cardinal.
The following definitions generalize that of the usual Baire and Cantor spaces (which correspond to the case \( \kappa = \omega \)), and of their Borel and analytic subsets.

\begin{defin} \label{def:boundedtopology}
The \emph{generalized Baire space} is the space \( \pre{\kappa}{\kappa} \)  equipped with the (\emph{bounded}) \emph{topology} \( \tau_b \), which is generated by the sets of the form
\begin{equation} \label{eq:basicofbounded}
\Nbhd_p = \{ x \in \pre{\kappa}{\kappa} \mid p \subseteq x \}
 \end{equation}
for \( p \in \pre{<\kappa}{\kappa} \).

The \emph{generalized Cantor space} \( \pre{\kappa}{2} \) is the closed subspace of \( \pre{\kappa}{\kappa} \) consisting of functions taking values in \( 2 = \{ 0,1 \} \).
\end{defin}

For the sake of simplicity, we will develop our theory of \( \kappa^+ \)-Borel sets for subspaces \( X \) of \( \pre{\kappa}{\kappa} \) (endowed with the relativization of the bounded topology \( \tau_b \)), but all definitions and results straightforwardly generalize to their homeomorphic copies.
\begin{defin}
Let \( X \subseteq \pre{\kappa}{\kappa} \) be endowed with the relative topology. A set \( A \subseteq X \) is called \emph{\( \kappa^+ \)-Borel} if it belongs to the \( \kappa^+ \)-algebra generated by the topology of \( X \). The collection of \( \kappa^+ \)-Borel subsets of \( X \) is denoted by \( \mathbf{Bor}(\kappa,X) \). 
\end{defin}

When \( \kappa \) is clear from the context we drop it from both the terminology and the notation above.
As in the classical case, (\(\kappa^+ \)-)Borel sets can be stratified into a hierarchy according to the following recursive definition:
\begin{align*}
\boldsymbol{\Sigma}^0_1(\kappa,X) &= \{ U \subseteq X \mid U \text{ is open} \} &&&
\boldsymbol{\Pi}^0_1(\kappa,X) &= \{ C \subseteq X \mid C \text{ is closed} \} \\
\boldsymbol{\Sigma}^0_\alpha(\kappa,X) &= \left\{ \bigcup_{\gamma<\kappa} A_\gamma \mid A_\gamma\in \bigcup_{1 \leq \beta < \alpha} \boldsymbol{\Pi}^0_\beta(\kappa,X) \right\} &&&
\boldsymbol{\Pi}^0_\alpha(\kappa,X) &= \left\{ X \setminus A  \mid A \in \boldsymbol{\Sigma}^0_\alpha(\kappa,X) \right\} 
\end{align*}
We also set \( \boldsymbol{\Delta}^0_\alpha(\kappa,X) = \boldsymbol{\Sigma}^0_\alpha(\kappa,X) \cap \boldsymbol{\Pi}^0_\alpha(\kappa,X) \), and we again drop \( \kappa \) from the notation whenever this is not a source of confusion.  
As shown in~\cite[Proposition 4.19]{AMR16},  for \( 1 \leq \alpha < \beta \)
\[ 
\boldsymbol{\Delta}^0_\alpha(X) \subsetneq \boldsymbol{\Sigma}^0_\alpha(X), \boldsymbol{\Pi}^0_\alpha(X) \subsetneq \boldsymbol{\Delta}^0_\beta(X),
 \] 
and moreover
\[ 
\mathbf{Bor}(X) = \bigcup_{1 \leq \alpha < \kappa^+} \boldsymbol{\Sigma}^0_\alpha( X) = \bigcup_{1 \leq \alpha < \kappa^+} \boldsymbol{\Pi}^0_\alpha( X) = \bigcup_{1 \leq \alpha < \kappa^+} \boldsymbol{\Delta}^0_\alpha( X).
 \] 
(When \( \kappa^{<\kappa} = \kappa \) one can use the classical arguments as in~\cite[Theorem 22.4]{Kec95}; otherwise a different proof is required.)
Notice also  that if \( \boldsymbol{\Gamma}(X) \) is one of \( \boldsymbol{\Sigma}^0_\alpha(X) \), \( \boldsymbol{\Pi}^0_\alpha(X) \), or \( \mathbf{Bor}(X) \), then for every \( A \subseteq X \)  
\[ 
A \in \boldsymbol{\Gamma}(X) \IFF   A = A' \cap X \text{ for some } A' \in \boldsymbol{\Gamma}(\pre{\kappa}{\kappa}).
\]
If \( A \in \mathbf{Bor}(X) \), the smallest ordinal \( 1 \leq \alpha < \kappa^+ \) such that \( A  \in \boldsymbol{\Sigma}^0_\alpha(X) \cup \boldsymbol{\Pi}^0_\alpha(X) \) is called the \emph{Borel rank}
 of \( A \) and denoted by \( \mathrm{rk}_B(A) \). To simplify some of the statements (and proofs) below, with a little abuse of notation we set \( \boldsymbol{\Sigma}^0_0(X) =  \boldsymbol{\Pi}^0_0(X) = \boldsymbol{\Delta}^0_1(X) \) and \( \mathrm{rk}_B(A) = 0 \) if \( A \in \boldsymbol{\Delta}^0_1(X) \).
 
 The symbol \( \boldsymbol{\Gamma}(X) \) will denote an arbitrary class of the form \( \boldsymbol{\Sigma}^0_\alpha(X) \), \( \boldsymbol{\Pi}^0_\alpha(X) \), or \( \boldsymbol{\Delta}^0_\alpha(X) \). In particular, we say that ``\( A  \) is \( \boldsymbol{\Gamma} (X)\)'' if \( A \in \boldsymbol{\Gamma}(X) \), and that ``\( A \) is a true \( \boldsymbol{\Gamma}(X) \) set'' if
  \( A \in \boldsymbol{\Gamma}(X) \) but it does not belong to any other class as above properly contained in \( \boldsymbol{\Gamma}(X) \). All these notions strictly depend on the ambient space \( X \). Nevertheless, when \( X\) is clear from the context we will remove any reference to it in all the terminology and notation above. This convention will be systematically applied when dealing with an equivalence relation \( E \) on some space \( X \), that is, when discussing the complexity of \( E \) we will always tacitly refer to its ambient space \( X \times X \).
  
\subsection{Borel codes}

Similarly to what happens in the classical case~\cite{Bl81}, \( \kappa^+ \)-Borel sets can be characterized via certain games on well-founded trees which essentially code how the given set is constructed from the clopen sets using the operations of \( \kappa \)-unions and \( \kappa \)-intersections.

	A \textit{tree}  $\Tr = (T, \leq)$ is a (nonempty) partial order with exactly one minimal element, called \textit{root}, and in which the set $\pred_\Tr(p) = \{q\in T\mid q <  p\}$ of predecessors of any \( p \in T \) is a finite linear order.  The elements of a tree are called \textit{nodes}.  The \textit{height} of a node $p\in \Tr$ is the order type (equivalently, the cardinality) of $\pred_{\Tr}(p)$. A \textit{leaf} is a terminal node, i.e.\ a node \( p \in \Tr \) such that \( p \not <  q \) for every \( q \in \Tr \). The tree \( \Tr \) is \emph{well-founded} if  it contains no infinite chain. In this case, we can recursively define the \emph{rank} \( \varrho_\Tr(p) \) of a node \( p \in \Tr \) as follows:
\begin{itemize}
	\item all leaves have rank $0$;
	\item if \( p \) is not a leaf, then $\varrho_\Tr(p)=\sup\{\varrho_\Tr(q)+1\mid p < q \in \Tr\}$.
\end{itemize}
The rank of the well-founded tree \( \Tr \) is \( \varrho(\Tr) = \varrho_\Tr(r) +1 \), where \( r \) is the root of \( \Tr \).
Notice that if \( |\Tr| \leq \kappa \) then $\varrho(\Tr)$ is always a successor ordinal smaller than \( \kappa^+ \). 

Particularly important examples of trees are the \emph{trees of finite sequences} over a set \( A \), namely, \( \Tr \subseteq \pre{<\omega}{A} \) which are closed under initial segments and ordered by end-extensions. If \( S \subseteq \pre{<\omega}{A} \), the \emph{tree generated by \( S \)} is 
\[ 
\Tr(S) = \{ t \in \pre{<\omega}{A} \mid t \subseteq s \text{ for some } s \in S \}.
\]
Notice that for such a tree \( \Tr \) we have that its root is \( \emptyset \) and \( |\Tr| \leq \max \{ \aleph_0,|A| \} \), whence if \( A \) is infinite and \( \Tr \) is well-founded, then \( \varrho(\Tr) < |A|^+ \). 
                                                                                                                                                                                                                                  
Canonical examples of well-founded trees of sequences are the sets \( \Tr_\alpha \) of all strictly decreasing sequences of ordinals \( < \alpha \), for \( \alpha \) any ordinal (so $\Tr_0$ is the singleton containing the empty sequence).
It is well known that such trees are universal among well-founded trees of size \( \kappa \), that is, every \( \kappa \)-sized well-founded tree embeds in \( \Tr_\alpha \) for some \( \alpha < \kappa^+ \). The next result makes explicit the dependence of such an \(\alpha\) from the rank of the tree under consideration.

\begin{lemma}\label{trees}
Let \( \kappa \) be an infinite cardinal. Every well-founded tree \( \Tr \) of size \( \leq \kappa \) and rank \( \beta +1 < \kappa^+ \) can be embedded into $\Tr_{\kappa \cdot \beta}$. Moreover, for every  \( \beta < \kappa^+ \) there exists a tree of size \( \leq \kappa \) and rank \( \beta+1 \)  which does not embed in any \( \Tr_{\alpha} \) with \( \alpha < \kappa \cdot \beta \).
\end{lemma}

\begin{proof}
By induction on \( \beta < \kappa^+ \). If \( \beta = 0 \), then \( \Tr \) consists only of its root, and thus it is isomorphic to \( \Tr_{\kappa \cdot 0 } = \Tr_0  = \{ \emptyset \} \). So let us assume that \( \beta > 0 \) (it makes no difference whether $\beta$ is successor or limit). In this case, the $\leq\kappa$-many immediate successors $\{ p_i \mid i < I \}$, \( I \leq \kappa \), of the root \( r \) of \( \Tr \) are in turn roots of the trees $\Tr^i = \{ q \in \Tr \mid p_i \leq q \}$, which are necessarily of rank $\leq\beta$. By inductive hypothesis, there are embeddings $\psi_i \colon \Tr^i \to \Tr_{\kappa \cdot \gamma_i}$ for some $\gamma_i<\beta$ (where if \( \beta = \gamma+1 \) we may have \( \gamma_i = \gamma \) for all \( i < I \)). Then the function \( \psi \colon \Tr \to \Tr_{\kappa \cdot \beta} \) defined by letting \( \psi(r) = \emptyset \) and \( \psi(q) \) be the sequence consisting of \( \kappa \cdot \gamma_i + i \) followed by \(\psi_i(q) \), where \( i < I \) is the unique index for which \( q \in \Tr^i \), is clearly a well-defined embedding.

The second part of the statement is again proved by induction on  \( \beta < \kappa^+ \). The basic 
case  \( \beta = 0 \) is trivial. Now assume that \( \beta = \gamma+1 \), and let \( \Tr' \) be a tree of size 
\( \leq \kappa \) and rank \( \gamma+1  \) which does not embed into any \( \Tr_\alpha \) for 
\( \alpha < \kappa \cdot \gamma \). Let \( \Tr \) be obtained by appending \( \kappa \)-many copies of 
\( \Tr' \) to a common root \( r \), and let \( p_i \), \( i < \kappa \), be an enumeration of the immediate 
successors of \(  r \) in \( \Tr \). Notice that \( \Tr \) has size \( \kappa \) and rank 
\( \varrho(\Tr') + 1 = \beta+1 \). Towards a contradiction, let 
\( \alpha < \kappa \cdot \beta = \kappa \cdot \gamma + \kappa \) be such that there is an embedding 
\( f \) of \(  \Tr \) into \( \Tr_\alpha \). Each \( f(p_i) \) is a nonempty sequence with some last element \( \alpha_i < \alpha \), so that the cone of \( \Tr_\alpha \) above \( f(p_i) \) is isomorphic to \( \Tr_{\alpha_i} \). Since there are \( \kappa \)-many \( p_i \)'s and 
\( \alpha < \kappa \cdot \gamma + \kappa \), by a cardinality argument one can check that
there is \( \bar\imath \in \kappa \) such that 
\( \alpha_{\bar\imath} < \kappa \cdot \gamma \). But then  the restriction of \( f \) to the cone of \( \Tr \) above \( p_{\bar\imath} \), which is isomorphic to \( \Tr' \), would 
yield an embedding of \( \Tr' \) into \( \Tr_{\alpha_{\bar\imath}} \), contradicting the choice of \( \Tr' \). The limit case is similar, the only complication being that we cannot start from a single tree of rank $\beta$ (since the rank of a tree is always a successor ordinal). This is fixed by choosing a sequence \( (\beta_i)_{i < \mathrm{cof}(\beta)} \) cofinal in \( \beta \) and, for each \( i < \mathrm{cof}(\beta) \), a tree \( \Tr^i \) of size \( \leq \kappa\) and rank \( \beta_i +1 \) which cannot be embedded in any \( \Tr_\alpha \) for \( \alpha < \kappa \cdot \beta_i \), and then considering the tree obtained by appending all these \( \Tr^i \) to a common root. 
\end{proof}

\begin{remark} \label{rmk:trees}
The proof of the first part of Lemma~\ref{trees} actually yields that each tree \( \Tr \) of size \( \leq \kappa\) and rank \( \beta+1 < \kappa^+ \) can be embedded into \( \Tr_{\kappa\cdot \beta} \) in the following strong sense: \( \Tr \) is isomorphic to a subtree of \( \Tr_{\kappa \cdot \beta} \) \emph{closed under initial segments}.
\end{remark}

We now present the games which characterize the \( \kappa^+ \)-Borel subsets of \( X \subseteq \pre{\kappa}{\kappa} \).  Let \( \Tr \) be a well-founded tree of size \( \leq \kappa \), let \( \ell \) be a labeling function sending the leaves of \( \Tr \) to clopen subsets of \( X \), and let \( x \) be an element of \( X \). The game \( G(\Tr, \ell, x) \) is played by two players \( \1 \) and \( \2 \) on the tree \( \Tr \) as follows. Player \( \1 \) starts playing an immediate successor of the root  of \( \Tr \), and afterwards \( \1 \) and \( \2 \) take turns in picking an immediate successor in \( \Tr \) of the opponent's previous move. Since \( \Tr \) is well-founded, after a finite number of turns a leaf \( p \) will be selected, so that the game cannot continue from that point on: when this happens, we have that \( \2 \) won the run if and only if \( x \in \ell(p) \), otherwise \( \1 \) won. Winning strategies for \( \1 \) and \( \2 \) are defined as usual, and we write \( \2 \wins G(\Tr, \ell, x) \) if player \( \2 \) has a winning strategy in such game. 

\begin{remark}
The set of all possible runs in \( G(\Tr,\ell,x) \) only depends on \( \Tr \), while \( \ell \) and \( x \) are involved only in the definition of the winning condition.
\end{remark}

A pair consisting of a well-founded tree \( \Tr \) of size \( \leq \kappa \) and a labeling function \( \ell \) as above will be called a \emph{\( \kappa^+ \)-Borel code}. Given such a code \( (\Tr,\ell) \), we let
\[ 
B(\Tr,\ell) = \{ x \in X \mid \2\wins G(\Tr, \ell, x) \}
 \] 
be the set \emph{coded by \( (\Tr, \ell) \)}.	

It is well known that a set \( A \subseteq X \) is \( \kappa^+ \)-Borel if and only if there is a \( \kappa^+ \)-Borel code \( (\Tr, \ell ) \) for it. In the next result we sharpen this by relating the Borel rank \( \mathrm{rk}_B(A) \) of \( A\) to the rank \( \varrho(\Tr) \) of \( \Tr \).  

\begin{theorem}\label{orig1}
Let \( \kappa \) be an infinite cardinal,  \( X \subseteq \pre{\kappa}{\kappa} \), and   $ \alpha<\kappa^+$.  Given a set \( A \subseteq X \), we have that $ A \in \boldsymbol{\Pi}_\alpha^0(X)$ if and only if $A = B(\Tr,\ell)$ for some \( \kappa^+ \)-Borel code \( (\Tr, \ell) \) with $\varrho(\Tr)\leq\alpha+1$.
\end{theorem}

Notice that since $\boldsymbol{\Sigma}_\alpha^0(X) \subseteq\boldsymbol{\Pi}_{\alpha+1}^0(X)$ for every $ \alpha<\kappa^+$, this also gives us a \( \kappa^+ \)-Borel code \( (\Tr, \ell) \) of any given set  $B \in \boldsymbol{\Sigma}_\alpha^0(X) \setminus\boldsymbol{\Delta}_\alpha^0(X)$ with $\varrho(\Tr) = \alpha+2$, while there cannot be a \( \kappa^+ \)-Borel code for \( B \) whose tree has rank \( < \alpha + 2 \) (unless \( \alpha = 0 \)).

\begin{proof}
First, we are going to show that every Borel set $B\in\boldsymbol{\Pi}_\alpha^0(X)$ is coded by some \( (\Tr, \ell) \) with \( \varrho(\Tr) \leq\alpha+1\).  We work by induction on $\alpha< \kappa^+$.
Assume first \( \alpha = 0 \), i.e.\ \( B \in \boldsymbol{\Pi}^0_0(X) = \boldsymbol{\Delta}^0_1(X) \): then \( B = B(\Tr,\ell) \) where \( \Tr \) consists just of its root (so that \( \varrho(\Tr) = 1 \)) and \( \ell(r) = B \). Assume now \( \alpha = 1 \). Then $B\in\boldsymbol{\Pi}_1^0(X)$, so that $B=\bigcap_{i<\kappa} B_i$ with $B_i$ clopen. Let $\Tr$ be the tree consisting of a root \( r \) together with \( \kappa \)-many immediate successors \( p_i \) (\( i<\kappa \)) of it, and let $\ell$ be the labeling function defined by $\ell(p_i) =  B_i$. Then \( \varrho(\Tr) = 2 \) and $B=B(\Tr, \ell)$.
Finally, let \( \alpha > 1 \). We have that $B=\bigcap_{i<\kappa} B_i$ for some $B_i\in\boldsymbol\Sigma_{\alpha_{i}}^0(X)$ with $1 \leq \alpha_i<\alpha$, and in turn $B_i=\bigcup_{j<\kappa}B_{i,j}$ for  some $ B_{i,j}\in\boldsymbol\Pi_{\alpha_{i,j}}^0(X)$ with $ \alpha_{i,j} < \alpha_i \). By the inductive hypothesis, \( B_{i,j} = B(\Tr_{i,j}, \ell_{i,j}) \) with \( \varrho(\Tr_{i,j}) \leq \alpha_{i,j}+1 \). Let us consider the tree $\Tr$ obtained by appending to each of the \( \kappa \)-many distinct successors \( p_i \) of its root \( r \) the trees \( \Tr_{i,j} \) (that is, the root \( p_{i,j} \) of each $\Tr_{i,j}$ is a distinct immediate successor of $p_i$). Notice that by construction
\( \varrho_{\Tr}(p_{i,j}) = \varrho_{\Tr_{i,j}}(p_{i,j}) \), whence \( \varrho_{\Tr}(p_{i,j}) \leq \alpha_{i,j} \) because \( \varrho_{\Tr_{i,j}}(p_{i,j})+ 1 = \varrho(\Tr_{i,j}) \leq \alpha_{i,j}+1 \). 
 By definition of rank, 
\[ 
\varrho_\Tr(p_{i})= \sup \{ \varrho_{\Tr}(p_{i,j})+1 \mid j < \kappa \} \leq \sup \{ \alpha_{i,j}+1 \mid j < \kappa \} \leq \alpha_i,
\] 
whence $\varrho_\Tr(r) = \sup \{ \varrho_\Tr(p_i)+1 \mid i < \kappa \} \leq \sup \{ \alpha_i+1 \mid i < \kappa \} \leq \alpha\) and \(\varrho(\Tr) \leq \alpha+1\). Define now a labeling function $\ell$ on the leaves of \( \Tr \) as follows. By construction, for each leaf \( p \) of \( \Tr \) there is a unique pair \( i,j \) of ordinals \( < \kappa \) such that \( p_{i,j} \leq p \): set \( \ell(p) = \ell_{i,j}(p) \). We claim that \( B = B(\Tr,\ell) \). In fact, consider a run in \( G(\Tr,\ell,x) \). In the first turn \( \1 \) will pick some \( p_{\bar\imath} \), and \( \2 \) will respond by picking some \( p_{\bar\imath,\bar\jmath} \). After these two moves, the rest of the run will be equivalent to a run in the game $G(\Tr_{\bar\imath,\bar\jmath},\ell_{\bar\imath,\bar\jmath},x)$. Thus we have $\2\wins G(\Tr,\ell,x)$ if and only if  for every $i<\kappa$ there is $j<\kappa$ such that $\2\wins G(\Tr_{ij},\ell_{ij},x)$, whence 
\[
B(\Tr,\ell) = \bigcap_{i < \kappa} \bigcup_{j < \kappa} B(\Tr_{i,j}, \ell_{i,j}) = \bigcap_{i < \kappa} \bigcup_{j < \kappa} B_{i,j} = \bigcap_{i< \kappa} B_i = B.
\] 

Conversely, we now prove that if $B = B(\Tr,\ell)$  with  $\varrho(\Tr) \leq\alpha+1$,  then \mbox{$B \in \boldsymbol{\Pi}_\alpha^0(X)$.} The proof is again by induction on $\alpha < \kappa^+$. 
If \(\alpha = 0 \), i.e.\ \( \varrho(\Tr) = 1 \), then \( \Tr \) consists only of its root \( r \) and 
\[
B= B(\Tr,\ell) = \ell(r) \in \boldsymbol{\Delta}^0_1(X) = \boldsymbol{\Pi}^0_0(X). 
\]
Assume now \( \alpha = 1 \).  Since we already dealt with the case \( \alpha = 0 \), we may assume \( \varrho(\Tr) = 2 \). Let $\{p_i\mid i<I\}$, for a suitable \( I \leq \kappa \), be the set of immediate successors
 of the root $r$ of \( \Tr \), so that \( \Tr \) contains no other nodes.
Then 
\[
B=B(\Tr,\ell) = \bigcap_{i<I}\ell(p_i) \in \boldsymbol{\Pi}^0_1(X).
\]

Finally, let \( \alpha > 1 \). We may assume $\varrho(\Tr) \geq 3$. Let $\{p_i\mid i<I\}$ be the set of immediate successors of the root $r$, and, for each \( i < I \), let $\{ p_{i,j} \mid j < J_i \}$ be the set of immediate successors of $p_i$ in \( \Tr \) (for suitable%
\footnote{Since we assumed $\varrho(\Tr) \geq 3$, we have that \( I > 0 \), while possibly \( J_i = 0 \) for some, but not all, \( i < I \).}
 \( I, J_i \leq \kappa \)). Finally, let \( \Tr_{i,j} \) be the subtree of \( \Tr \) with domain \( \{ p \in \Tr \mid p_{i,j} \leq p \} \), and let \( \ell_{i,j} \) be defined on the leaves \( p \) of \( \Tr_{i,j} \) by setting \( \ell_{i,j}(p) = \ell(p) \) (notice that \( p \) is a leaf of \( \Tr_{i,j} \) if and only if \( p \) is a leaf of \( \Tr \) and \( p \in \Tr_{i,j} \)). By construction, \( B = B(\Tr,\ell) = \bigcap_{i < I} \bigcup_{j < J_i} B(\Tr_{i,j},\ell_{i,j} ) \). Moreover, since \( \varrho(\Tr) \leq \alpha +1 \), we get \( \varrho_\Tr(r) \leq \alpha \), and by definition of rank \( \varrho_\Tr(p_{i,j}) < \varrho_\Tr(p_i)	< \varrho_\Tr(r) \) for all relevant \( i,j \). It follows that
 \[ 
\varrho(\Tr_{i,j}) \leq \varrho_{\Tr}(p_i) < \alpha.
 \] 
By inductive hypothesis, this implies that
\[ 
B(\Tr_{i,j},\ell_{i,j}) \in \bigcup_{\beta < \varrho_\Tr(p_i)} \boldsymbol{\Pi}^0_\beta(X),
 \] 
whence 
\[ 
\bigcup_{j < J_i} B(\Tr_{i,j},\ell_{i,j}) \in \boldsymbol{\Sigma}^0_{\varrho_\Tr(p_i)}(X) \subseteq \bigcup_{\beta < \alpha} \boldsymbol{\Sigma}^0_\beta(X)
\]
for all \( i < I \), which in turn implies
\[ 
\bigcap_{i < I} \bigcup_{j < J_i} B(\Tr_{i,j},\ell_{i,j} ) \in \boldsymbol{\Pi}^0_\alpha(X),
 \] 
 as desired.
\end{proof}

\begin{remark}
It is clear from the proof above that we still obtain \( \kappa^+ \)-Borel sets if we modify the definition of \( \kappa^+ \)-Borel codes by allowing the labeling function to take arbitrary \( \kappa^+ \)-Borel sets as values. However, the Borel rank of the coded set would in this case depend on the Borel ranks of the sets used as labels.
\end{remark}

We also notice that one can code all \( \kappa^+ \)-Borel sets by using only the canonical well-founded trees \( \Tr_\alpha \) to form codes: in a sense this  shows that the relevant information in a \( \kappa^+ \)-Borel code actually relies on the labeling function together with the rank of the tree, but not on the specific tree itself.

\begin{cor}\label{improve}
	Every $\kappa^+$-Borel set admits a $\kappa^+$-Borel code of the form $(\Tr_\alpha, \ell)$ for some $\alpha<\kappa^+$. More precisely, if $B\in\boldsymbol{\Pi}_\alpha^0(X)$, then $B = B(\Tr_{\kappa \cdot \alpha},\ell)$ for some labeling function \( \ell \).
\end{cor}
\begin{proof}
	By Theorem~\ref{orig1} there is a Borel code \( (\Tr', \ell') \) for \( B \) with \(\varrho(\Tr')\leq\alpha+1 \). By Lemma~\ref{trees} and Remark~\ref{rmk:trees}, we may assume without loss of generality that \( \Tr' \) is a subtree of \(  \Tr_{\kappa\cdot \alpha} \) closed under initial segments. 
Let \( \ell \) be the labeling function defined on the leaves \( p \) of \( \Tr_{\kappa \cdot \alpha} \) as follows. Let \( p' \) be the largest node such that \( p' \in \Tr' \) and \( p' \leq p \). We distinguish three cases:
	\begin{enumerate-(1)}
		\item \label{sameleaf} if \( p' \) is a leaf of \( \Tr' \), then set \( \ell(p) = \ell'(p') \);

		\item \label{emptyleaf} if \(  \pred_{\Tr_{\kappa \cdot \alpha}}(p') = \pred_{\Tr'}(p') \) has an odd number of elements, then set \( \ell(p) =  \emptyset \);

		\item \label{fullleaf} if \(  \pred_{\Tr_{\kappa \cdot \alpha}}(p') = \pred_{\Tr'}(p') \) has an even number of elements, then set \mbox{\( \ell(p) =  X \).}
	\end{enumerate-(1)}
	We claim that \( B(\Tr_{\kappa \cdot \alpha}, \ell) = B(\Tr',\ell') \), whence \( B = B(\Tr_{\kappa \cdot \alpha}, \ell) \).  
	
	Indeed, let \( x \in X \). If \( \2 \wins G(\Tr_{\kappa \cdot \alpha}, \ell, x) \), then by~\ref{emptyleaf} his winning strategy never involves playing a node outside \( \Tr' \) unless either \( \1 \) already did or in a  previous turn a leaf \( p' \) of \( \Tr' \) was reached, in which case any leaf \( p \) that will be reached at the end of the run will be such that \( \ell(p) = \ell'(p') \) by~\ref{sameleaf}. Thus the restriction of any winning strategy of \( \2 \) to \( \Tr' \) actually witnesses \( \2 \wins G(\Tr', \ell', x ) \), since obviously in the restricted game \( \1 \) always plays inside \( \Tr' \). 
	
	Conversely, a winning strategy for \( \2 \) in \( G(\Tr',\ell',x) \) can be converted into a winning strategy for \( \2 \) in \( G(\Tr_{\kappa \cdot \alpha}, \ell ,x) \) as follows. As long as \( \1 \) is playing nodes in \( \Tr' \) which are not leaves of \( \Tr' \), player \( \2 \) follows his strategy in \( G(\Tr',\ell',x) \) (notice that in this case the node played by \( \2 \) will be in \( \Tr' \) as well). If \( \1 \) plays for the first time a node outside \( \Tr' \), then \( \2 \) can make random moves from that point on because by~\ref{fullleaf} any leaf \( p \) of \( \Tr_{\kappa \cdot \alpha} \) that will be reached will satisfy \( \ell(p) = X \), whence \( x \in \ell(p) \) trivially. In the remaining case, i.e.\ when a leaf \( p' \) of \( \Tr' \) has been reached by either \( \1 \) or \( \2 \), player \( \2 \) can again make random moves from that point on because by~\ref{sameleaf} any leaf \( p \) of \( \Tr_{\kappa \cdot \alpha} \) that will be reached at the end of the run will be such that \( \ell(p) = \ell'(p') \), whence \( x \in \ell(p) \) because \( p' \) was reached following the winning strategy of \( \2 \) in \( G(\Tr',\ell',x) \).
\end{proof}

\begin{remark}
Corollary~\ref{improve} allows us to reformulate the games coding \( \kappa^+ \)-Borel sets as follows. Given \( \alpha < \kappa^+ \), a labeling function \( \ell \colon \Tr_\alpha \to \boldsymbol{\Delta}^0_1(X) \), and a point \( x \in X \), the game \( G_\alpha(\ell,x) \) is played as follows. Player \( \1 \) start by choosing some ordinal \( \alpha_0 < \alpha \) and player \( \2 \) responds with some \( \alpha_1 < \alpha_0 \). Then \( \1 \) chooses some \( \alpha_2 < \alpha_1 \) while \( \2 \) chooses \( \alpha_3 < \alpha_2 \). They continue in this way until \( 0 \) is reached, at which point we say that \( \2 \) wins if and only if \( x \in \ell(\langle \alpha_0, \alpha_1, \dotsc, 0 \rangle) \). It turns out from what we proved above that \( A \subseteq X \) is \( \kappa^+ \)-Borel if and only if there are \( \alpha <  \kappa^+ \) and \( \ell \colon \Tr_\alpha \to \boldsymbol{\Delta}^0_1(X) \) such that \( A \) is the set of those \( x \in X \) for which \( \2 \) has a winning strategy in \( G_\alpha(\ell,x) \).
\end{remark}

\subsection{Codes for \( \kappa \)-sized structures} \label{subsec:codes}

Our use of generalized descriptive set theory is mainly concerned with spaces of codes for first-order structures of size \( \kappa \). For the sake of simplicity we will consider only  finitary relational%
\footnote{This is not a true limitation, as functions can be dealt with through their graphs, and constants can be construed as \( 0 \)-ary functions.} 
languages \( \L = \{ R_i \mid i < I \} \), where \( I \leq \kappa \) and \( R_i \) is a relation symbol of arity \( n_i \). Up to isomorphism, we can assume without loss of generality that every \( \kappa \)-sized \mbox{\( \L \)-structure} has domain \( \kappa \), hence it can be coded through the characteristic functions of its predicates. Therefore we can regard
\[ 
\Mod^\kappa_\L = \prod_{i \leq I} \pre{\left(\pre{n_i}{\kappa}\right)}{2}
 \] 
as the space of (codes of) all \( \kappa \)-sized \( \L \)-structures. It is natural to equip this space with the \emph{logic topology}, i.e.\ with the topology generated by the sets
\begin{equation} \label{eq:basicofmod}
\Nbhd_\Q=\left\{\M\in\Mod^\kappa_\L \mid  \Q \text{ is a substructure of the \( \L' \)-reduct of } \M \right\},
\end{equation}
with
 \( \L' \) varying over subsets of \( \L \) of size \( < \kappa \), and \( \Q \) varying over the \( \L' \)-structures with  domain bounded%
\footnote{In particular, such a \( \Q \) has size \( < \kappa \). When \( \kappa \) is regular, these two conditions become equivalent: the domain of \( \Q \) is bounded in \( \kappa \) if and only if it has size \( < \kappa \).}
 in \( \kappa \). Clearly, if \( \L \) already has size \( < \kappa \) we can avoid any reference to \( \L' \) and just let \( \Q \) vary over \( \L \)-structures with domain bounded in \( \kappa \). It is an easy exercise to show that \( \Mod^\kappa_\L \) is homeomorphic to \( \pre{\kappa}{2} \), so that we can speak of (\( \kappa^+ \)-)Borel subsets of \( \Mod^\kappa_\L \) and of their Borel ranks. Given some theory $T$, we denote by $\mathrm{Mod}^\kappa_T$ the space of $\kappa$-sized models of $T$ with the subspace topology induced by $\mathrm{Mod}^\kappa_\L$.

Recall that the infinitary logic \( \L_{\kappa^+ \kappa} \) is the extension of the usual first-order logic obtained by allowing conjunctions and disjunctions of length \( \leq \kappa \) and (simultaneous) quantifications over sequences of variables of length \( < \kappa \), while \( \L_{\infty \kappa} \) is the further extension of \( \L_{\kappa^+ \kappa} \) in which we also allow conjunctions and disjunctions of arbitrary (set-)size. 
For \( \sigma \) an \( \L_{\kappa^+ \kappa} \)-sentence, we set
\[ 
\Mod^\kappa_\upsigma = \{ \M \in \Mod^\kappa_\L \mid \M \models \upsigma \},
 \] 
 and we say that a set \( A \subseteq \Mod^\kappa_\L \) is \emph{axiomatized} by \( \upsigma \) if \( A = \Mod^\kappa_\sigma \).
Arguing as in the classical case \( \kappa = \omega \), when \( \kappa^{< \kappa} = \kappa \) there is a tight relation between the \( \kappa^+ \)-Borel subsets of \( \Mod^\kappa_\L \) closed under isomorphism and the subsets of \( \Mod^\kappa_\L \) that can be axiomatized within the logic \( \L_{\kappa^+ \kappa} \). 

\begin{theorem}\label{orig2}
Let \( \kappa \) be such that $\kappa^{<\kappa}=\kappa$. For every \( A \subseteq \Mod^\kappa_\L \) the following are equivalent:
\begin{enumerate-(i)}
\item
\( A \in \mathbf{Bor}(\Mod^\kappa_\L) \) and is \emph{closed under isomorphism} (i.e.\ \( \M \in A \) and \( \M \cong \N \) implies \( \N \in A \));
\item
\( A  \) is axiomatized by some \( \L_{\kappa^+ \kappa} \)-sentence \( \upsigma \).
\end{enumerate-(i)}
\end{theorem}

A full proof of this theorem can be found in~\cite[Theorem 4.1]{Vau74} or~\cite[Theorem 24]{FHK14}, and more refined versions of it are presented in~\cite[Section 8.2]{AMR16}. 

\begin{remark} \label{rmk:borelrankfirstordertheories} 
Inspecting the above mentioned proofs from~\cite{FHK14,AMR16}, it is not hard to see that if some \( A \subseteq \Mod^\kappa_\L \) is axiomatized by a first-order formula, then its Borel rank is finite. It follows that if \( A  = \Mod^\kappa_T = \bigcap_{\upsigma \in T} \Mod^\kappa_\upsigma  \) for some countable first-order theory \( T \), then \( A \) has Borel rank \( \leq \omega \).
\end{remark}

For what follows, we need to modify Theorem~\ref{orig2} in two directions:
\begin{enumerate-(a)}
\item
we need to ``relativize'' it to arbitrary subspaces of \( \Mod^\kappa_\L \) closed under isomorphism;
\item
we need a level-by-level version connecting the Borel rank of the set \( A \) to the quantifier rank of the \( \L_{\kappa^+ \kappa} \)-sentence axiomatizing it, as defined below.
\end{enumerate-(a)}

\begin{defin}
Let \( \upvarphi \) be an \( \L_{\infty \kappa} \)-formula. The \textit{quantifier rank} $R(\upvarphi)$ of $\upvarphi$ is defined by recursion on the complexity of $\upvarphi$ as follows:
	\begin{itemize}
		\item if $\upvarphi$ is atomic, then $R(\upvarphi)=0$;
		\item if $\upvarphi$ is of the form $\neg\uppsi$, then $R(\upvarphi)=R(\uppsi)$;
		\item if $\upvarphi$ is of the form $\bigwedge_{j \in J} \uppsi_j$ for some set \( J \), then $R(\upvarphi)=\sup_{i \in J}R(\uppsi_j)$;
		\item if $\upvarphi$ is of the form $\exists\bar x \,\uppsi$, then $R(\upvarphi)=R(\uppsi)+1$.
	\end{itemize}
\end{defin}

\begin{defin}
A set \( \ElCl \subseteq \Mod^\kappa_\L \) is called an \emph{invariant set} if it is closed under isomorphism. If \( \ElCl \) is of the form \( \Mod^\kappa_T \) for \( T \) a first-order theory (respectively,  of the form \( \Mod^\kappa_\upvarphi \) for \( \upvarphi \) an \( \L_{\kappa^+ \kappa} \)-sentence) we say that \( \ElCl \) is \emph{axiomatized} by \( T \) (respectively, by \( \upvarphi \)) and call it a \emph{first-order elementary class} (respectively, an \emph{\( \L_{\kappa^+ \kappa} \)-elementary class}).
\end{defin}

\begin{theorem}\label{lopezext}
Let \( \kappa \) be such that \( \kappa^{< \kappa} = \kappa \).
Let \( \ElCl \subseteq \Mod^\kappa_\L \) be an arbitrary invariant set. Then for every \( A \subseteq \ElCl \), we have that \( A \in \mathbf{Bor}(\ElCl) \) and is closed under isomorphism if and only if there is an \( \L_{\kappa^+ \kappa} \)-sentence \( \upsigma \) such that \( A = \Mod^\kappa_\upsigma \cap {\ElCl} \). Moreover, if \( A \) is \( \kappa^+ \)-Borel and closed under isomorphism, then \( \upsigma \) can be chosen so that \( R(\upsigma) \leq \max \{ \mathrm{rk}_B(A),1 \} \).
\end{theorem}

\begin{proof}
One direction is easy: if \( \upsigma \) is an \( \L_{\kappa^+ \kappa} \)-sentence, then \( \Mod^\kappa_\upsigma \in \mathbf{Bor}(\Mod^\kappa_\L)\) by Theorem~\ref{orig2}, hence \( A = \Mod^\kappa_\upsigma \cap \ElCl \in \mathbf{Bor}(\ElCl) \).

For the other direction, following the proof of Theorem~\ref{orig2} we first need the following claim. Let \( S_\kappa \subseteq \pre{\kappa}{\kappa} \) be the group of all permutations \( p \colon \kappa \to \kappa \) endowed with the relative topology. A basis for \( S_\kappa \) is given by the sets
\[ 
\Nbhd_u^{-1} = \{ p \in S_\kappa \mid p^{-1} \in \Nbhd_u \cap S_\kappa  \},
 \] 
where \( u \in \pre{< \kappa}{\kappa} \) and \( \Nbhd_u \) is the usual basic open set of \( \pre{\kappa}{\kappa} \) determined by \( u \) (clearly, \( \Nbhd_u^{-1} \neq \emptyset \) if and only if \( u \) is injective). Recall that if \( \kappa^{< \kappa} = \kappa \) then \( S_\kappa \) satisfies the (generalized) Baire category theorem: every nonempty open subset of \( S_\kappa \) is not \( \kappa \)-meager, i.e.\ it is not the union of \( \kappa \)-many nowhere dense sets (\cite[Theorem 6.12]{AMR16}). Consider the logic action of \( S_\kappa \) on \( \Mod^\kappa_\L \) defined by letting \( p \M \) be the unique \( \L \)-structure for which \( p \) is an isomorphism between \( \M \) and \( p \M \); notice that such action is continuous. Finally, given \( A \subseteq \Mod^\kappa_\L \) and \( u \in \pre{<\kappa}{\kappa} \), let \( A^{*u} \subseteq \Mod^\kappa_\L \) be the collection of those \( \M \) for which the set \( \{ p \in \Nbhd^{-1}_u \mid p \M \in A  \} \) is \( \kappa \)-comeager (i.e.\ the complement of a \( \kappa \)-meager set) in \( \Nbhd^{-1}_u \).

\begin{claim} \label{claim:axioms}
Let \( A \in \mathbf{Bor}(\Mod^\kappa_\L ) \). For every \( \beta < \kappa \) there is an \( \L_{\kappa^+ \kappa} \)-formula \( \upvarphi_\beta^A((x_i)_{i < \beta}) \) such that for every \( u \in \pre{\beta}{\kappa} \)
\begin{equation} \label{eq:axioms}
A^{*u} = \{ \M \in \Mod^\kappa_\L \mid \M \models \upvarphi^A_\beta[u] \},
 \end{equation}
where \( \upvarphi^A_\beta[u] \) is obtained by assigning  to each variable \( x_i \), \( i < \beta \), the element \mbox{\( u(i) \in \M \).} Moreover, if \( A \in \boldsymbol{\Pi}^0_\alpha(\Mod^\kappa_\L ) \) for some \(  \alpha < \kappa^+ \), then we can also get  \( R(\upvarphi^A_\beta) \leq \max \{ \alpha, 1 \} \).
\end{claim}

\begin{proof}[Proof of the claim]
A full proof of the first part of the claim can be found in the second part of the proof of~\cite[Theorem 24]{FHK14} or  in~\cite[Lemma 8.16]{AMR16} (indeed, Claim~\ref{claim:axioms} is the special case  where \( \lambda  = \mu = \kappa = \kappa^{< \kappa} \)). Here we just observe that the formulas \( \upvarphi^A_\beta \) constructed therein have the correct quantifier rank. We first consider four basic cases.
\begin{enumerate-(1)}
\item \label{axioms:case1}
Assume that \( A = \Nbhd_\Q \) is a basic open set of \( \Mod^\kappa_\L \). Let \( \gamma < \kappa \) be smallest such that \(  \Q \subseteq \gamma \), and let \( \uptheta((y_j)_{j < \gamma} ) \) be the (conjunction of the formulas in the) \( \L' \)-atomic diagram of \( \Q \), where \( \L' \) is the signature of \( \Q \) (i.e.\ the unique sublanguage of \( \L \) of size \( < \kappa \) such that \( \Q\) is an \( \L' \)-structure). Then \( \upvarphi^A_\beta \) is either
\[ 
\bigwedge_{j < \gamma} (x_j = y_j) \wedge \uptheta((y_j)_{j < \gamma} )
 \] 
or
\[ 
\forall (y_j)_{j < \gamma} \left ( \bigwedge_{i < \beta} (x_i = y_i) \wedge \bigwedge_{i<j<\gamma} (y_i \neq y_j) \to \uptheta((y_j)_{j < \gamma})  \right ),
 \] 
depending on whether \( \gamma \leq \beta \) or \( \beta < \gamma \). Notice that in both cases, \( R(\varphi^A_\beta) \leq 1 \).
\item \label{axioms:case2}
Assume now that \( A = \Mod^\kappa_\L \setminus \Nbhd_\Q \). This case was not explicitly dealt with in~\cite{FHK14,AMR16}. However, inspecting the proof of the previous case one easily sees that it is enough to systematically replace \( \uptheta((y_j)_{j < \gamma}) \) with \( \neg \uptheta((y_j)_{j < \gamma}) \) in the formulas appearing in~\ref{axioms:case1} to get the desired \( \upvarphi^A_\beta \). So also in this case \( R(\varphi^A_\beta) \leq 1 \).
\item \label{axioms:case3}
Assume now that \( A = \bigcap_{\delta < \kappa} A_\delta \) with \( \upvarphi^{A_\delta}_\beta \) witnessing~\eqref{eq:axioms} for the set \( A_\delta \) (and the same \( \beta < \kappa \)). Then it is enough to let \( \upvarphi^A_\beta \) be the formula \( \bigwedge_{\delta < \kappa} \upvarphi^{A_\delta}_\beta \). Notice that \( R(\varphi^A_\beta) = \sup_{\delta < \kappa} R(\upvarphi^{A_\delta}_\beta) \).
\item \label{axioms:case4}
Finally, assume that \( A = \Mod^\kappa_\L \setminus B \) with \( \upvarphi^B_\gamma \) witnessing~\eqref{eq:axioms}  for the set \( B \) and \emph{an arbitrary \( \gamma < \kappa \)}.  Then it is enough to let \( \upvarphi^A_\beta \) be the formula
\[ 
\bigwedge_{\beta < \gamma < \kappa}  \forall (y_j)_{j < \gamma} \left [ \left ( \bigwedge_{i < \beta} (x_i = y_i) \wedge \bigwedge_{i<j<\gamma} (y_i \neq y_j)  \right ) \to \neg \upvarphi^B_\gamma((y_j)_{j < \gamma}) \right ]
 \] 
Notice that \( R(\upvarphi^A_\beta) = \sup_{\beta < \gamma < \kappa} (R(\upvarphi^B_\gamma)+1) \).
\end{enumerate-(1)}
Using these facts, one can easily check by induction on \( \alpha < \kappa^+ \) that the second part of the statement is true. Indeed, for the basic cases \( \alpha = 0 \) or \( \alpha = 1 \) it is enough to observe%
\footnote{Notice that the bound on \( R(\upvarphi^A_\beta) \) cannot be improved when \( \alpha = 0 \): even in the simplest case of a (nontrivial) basic clopen set \( A \), we still have  \( R(\upvarphi^A_\beta)  = 1\) for small enough \( \beta \)'s.} 
that since we assumed 
\( \kappa^{< \kappa} = \kappa \), each \( A \in \boldsymbol{\Pi}^0_1(\Mod^\kappa_\L ) \) different from \( \Mod^\kappa_\L \) is 
of the form 
\( A = \bigcap_{\delta < \kappa} (\Mod^\kappa_L \setminus \Nbhd_{\Q_\delta}) \) for the appropriate 
basic open sets \( \Nbhd_{\Q_\delta} \), whence \( R(\upvarphi^A_\beta) \leq 1 \) by~\ref{axioms:case2} 
and~\ref{axioms:case3}. If instead \( A = \Mod^\kappa_\L \), then \( A = \Nbhd_\emptyset \) and hence we can conclude as well using~\ref{axioms:case1}. The limit case obviously follows from~\ref{axioms:case3}. Finally, the successor 
step follows from~\ref{axioms:case4} and~\ref{axioms:case3}, together with the fact that each 
\( A \in \boldsymbol{\Pi}^0_{\alpha+1}(\Mod^\kappa_\L) \) is by definition of the form 
\( A = \bigcap_{\delta< \kappa} (\Mod^\kappa_\L \setminus A_\delta ) \) with 
\( A_\delta \in \boldsymbol{\Pi}^0_\alpha(\Mod^\kappa_\L) \) for every \( \delta < \kappa \) (here we 
are also using  that \( R(\upvarphi^{A_\delta}_\gamma) \) is actually independent of 
\( \gamma \)).
\end{proof}
Assume now that \( A \in \mathbf{Bor}(\ElCl) \), with \( \mathrm{rk}_B(A) = \alpha \) for some \(  \alpha < \kappa^+ \), and let it be closed under isomorphism. 
Without loss of generality we may assume \( A \in \boldsymbol{\Pi}^0_\alpha(\ElCl) \). (Indeed, if 
\( A \in \boldsymbol{\Sigma}^0_\alpha(\ElCl) \) and \( \upsigma \) is the \( \L_{\kappa^+ \kappa} \)-sentence witnessing 
the theorem for \( \ElCl \setminus A \in \boldsymbol{\Pi}^0_\alpha(\ElCl) \), then \( \neg \upsigma \) is a witness  for \( A \).)  Let 
\( A' \in \boldsymbol{\Pi}^0_\alpha(\Mod^\kappa_\L) \) be such that \( A = A' \cap \ElCl \). Applying Claim~\ref{claim:axioms} 
to such \( A' \) with \( \beta = 0 \) we get an \( \L_{\kappa^+ \kappa } \)-sentence \( \upvarphi^{A'}_0 \) such that 
\( (A')^{* \emptyset} = \{ \M \in \Mod^\kappa_L \mid \M \models \upvarphi^{A'}_0 \} = \Mod^\kappa_{\upvarphi^{A'}_0} \): 
we claim
that \( \upvarphi^{A'}_0 \) is the \( \L_{\kappa^+ \kappa} \)-sentence \( \upsigma \) witnessing the theorem for $A$. We first prove that
\[ 
(A')^{* \emptyset} \cap \ElCl = A.
 \] 
In fact, if \( \M \in A \) then \( \{ p \in S_\kappa = \Nbhd^{-1}_\emptyset \mid p\M \in A' \} = S_\kappa \) because \( A = A' \cap \ElCl \) is closed under isomorphism, and \( S_\kappa \) is trivially \( \kappa \)-comeager in 
\( \Nbhd^{-1}_\emptyset  = S_\kappa \) because \( S_\kappa \) satisfies the (generalized) Baire category theorem; it follows 
that \( \M \in (A')^{* \emptyset} \cap \ElCl \). Conversely, if \( \M \in \ElCl \setminus A \) then 
\( \{ p \in \Nbhd^{-1}_\emptyset \mid p\M \in A' \} = \emptyset \) because both \( \ElCl \) and \( A = A' \cap \ElCl \) are closed under 
isomorphism; since \( \emptyset \) is trivially \( \kappa \)-meager (hence not \( \kappa \)-comeager) in 
\( \Nbhd^{-1}_\emptyset \), we conclude \( \M \notin (A')^{* \emptyset} \cap \ElCl \). 

Summing up, we have shown that \( A = \Mod^\kappa_{\upvarphi^{A'}_0} \cap \ElCl \), and hence we are done because \( \mathrm{rk}_B(A) = \alpha \) and  \( R(\upvarphi^{A'}_0) \leq \max \{ \alpha ,1 \} \) by the second part of Claim~\ref{claim:axioms}.
\end{proof}

\subsection{Equivalence relations}

Throughout the paper we will repeatedly use the following easy observations, often without explicitly mentioning them.

\begin{fact} \label{fct:fromertoclass}
Let \( E \) be an equivalence relation on a topological space \( X \). Then each \( E \)-equivalence class \( [x]_E \) is the continuous preimage of \( E \) through the continuous function \( f \colon X \to X^2 \) defined by setting \( f(y) = (y,x) \).
\end{fact}

\begin{fact} \label{fct:fromclasstoer}
Let \( E \) be an equivalence relation on \( X \). Then
\[ 
E = \bigcup \{ [x]_E \times [x]_E \mid x \in X \}
 \] 
and
\[ 
X^2 \setminus E = \bigcup \{ [x]_E \times [y]_E \mid x,y \in X \text{ and } x \not\mathrel{E} y \}.
 \] 
\end{fact}

\begin{prop}[Folklore] \label{prop:open}
Let \( X \) be any topological space, and \( E \) be an equivalence relation on \( X \). The following are equivalent:
\begin{enumerate-(a)}
\item \label{prop:open-a}
\( E \) is open;
\item \label{prop:open-b}
all \( E \)-equivalence classes are open;
\item \label{prop:open-c}
all \( E \)-equivalence classes are clopen;
\item \label{prop:open-d}
\( E \) is clopen.
\end{enumerate-(a)}
If the equivalent conditions above are satisfied for \( E \) and \( X \) has a basis of size \( \kappa \), then \( E \) has at most \( \kappa \)-many classes.
\end{prop}

\begin{proof}
Given \( x \in X \), let \( [x]_E \) be the \( E \)-equivalence class of \( E \). If \( E \) is open, then 
\( [x]_E \) is open as well by Fact~\ref{fct:fromertoclass}: this shows \ref{prop:open-a}~$\IMP$~\ref{prop:open-b}.
To show \ref{prop:open-b}~$\IMP$~\ref{prop:open-c} it is enough to observe that 
\[ 
X \setminus [x]_E = \bigcup \{ [y]_E \mid x,y \in X \text{ and } x \not\mathrel{E} y \} ,
\] 
while \ref{prop:open-c}~$\IMP$~\ref{prop:open-d} follows from Fact~\ref{fct:fromclasstoer}.
Since \ref{prop:open-d}~$\IMP$~\ref{prop:open-a} is obvious, we get that the four conditions \ref{prop:open-a}--\ref{prop:open-d} are equivalent to each other.

For the additional part, notice that if \( \mathcal{B} \) is a basis for \( X \) and \( x_0 \) is any element of \( X \), then the map \( f \colon \mathcal{B} \to X/E \) defined by
\[ 
f(B) =
\begin{cases}
[x]_E & \text{if } B \neq \emptyset \text{ and } B \subseteq [x]_E \\
[x_0]_E  & \text{otherwise}
\end{cases}
 \] 
is well-defined and surjective (because all \( E \)-equivalence classes are nonempty open sets).
\end{proof}

\begin{cor}
For any topological space there is no true \( \boldsymbol{\Sigma}^0_1 \) equivalence relation on \( X \). Moreover, if \( X \) is connected (e.g.\ \( X = \RR^n \)) than the unique (cl)open equivalence relation on it is the trivial one, that is \( E = X^2 \).
\end{cor}

Proposition~\ref{prop:open} shows that no equivalence relation can be a true open set; for all other possible complexities one can instead build an equivalence relation of exactly that complexity.

\begin{example}[Folklore] \label{xmp:allcomplexities}
Let \( X \) be a Hausdorff topological space and \( \boldsymbol{\Gamma} \) be one of the classes \( \boldsymbol{\Sigma}^0_\alpha \), \( \boldsymbol{\Pi}^0_\alpha \), or \( \boldsymbol{\Delta}^0_\alpha \), with the sole exception of \( \boldsymbol{\Sigma}^0_1 \). Suppose that there is a true \( \boldsymbol{\Gamma}(X) \) set \( A \). Define the equivalence relation \( E_A \) on \( Y =  X \times \{ 0,1 \} \) by setting for \( x,y \in X \) and \( i,j \in \{ 0,1 \} \) 
\[ 
(x,i) \mathrel{E_A} (y,j) \IFF (x = y ) \wedge (i = j \vee x \in A). 
 \] 
Then all \( E_A \)-equivalence classes have either \(1 \) or \( 2 \) elements (in particular, they are closed sets), and \( E_A \) is a true \( \boldsymbol{\Gamma} \) set.
\end{example}

\section{Categoricity from the topological viewpoint} \label{subsec:categoricity}

It is well known that every complete first-order theory $T$ has the joint embedding property, namely: for every pair \( \M_0 , \M_1 \) of models of \( T \)  there is \( \N \models T \) in which \( \M_0 \) and \( \M_1 \) jointly embed (see any classical model theory textbook, e.g. \cite{hodges1993model}). From this one can easily infer that \( T \) satisfies a cardinality-preserving version of such property.

\begin{lemma} \label{lem:jep}
Let \( \kappa \) be any infinite cardinal, \( T \) be a (not necessarily countable) complete first-order theory in a language \( \L \) of size \( \leq \kappa \), and let \( \M_0,\M_1 \in \Mod^\kappa_T \). Then there exists $\N\in\Mod^\kappa_T$ such that $\M_0$ and $\M_1$ both embed into it.
\end{lemma}

\begin{proof}
Use the joint embedding property to find some \( \widetilde{\N} \models T \) such that \( \M_0,\M_1 \) both embed into \( \widetilde{\N}\) (so that in particular \( \widetilde{\N} \) has size \( \geq \kappa \)), and fix embeddings \( f_0,f_1 \) witnessing this. Let \( \L' \) be the language obtained by adding to \( \L \) a new constant symbol \( a_\alpha \) for any \( \alpha \in \M_0 \) and a new constant symbol \( b_\beta \) for any \( \beta \in \M_1 \). Expand \( \widetilde{\N} \) to an \( \L' \)-structure \( \widetilde{\N}' \) by interpreting each \( a_\alpha \) in \( f_0(\alpha) \) and each \( b_\beta \) in \( f_1(\beta) \). Since \( \L' \) has size \( \leq \kappa \), by the downward L\"owenheim-Skolem theory there is an elementary substructure \( \N' \) of \( \widetilde{\N}' \) of size \( \kappa \). By the choice of the interpretations of \( a_\alpha, b_\beta \) in \( \widetilde{\N}' \), it follows that the \( \L \)-reduct \( \N \) of \( \N' \) is (isomorphic to a structure) in \( \Mod^\kappa_T \) and both \( \M_0 \) and \( \M_1 \) embed into it.\end{proof}

\begin{prop} \label{prop:density}
Let \( \kappa \) be any infinite cardinal, \( T \) be a (not necessarily countable) complete first-order theory in a language \( \L \) of size \( \leq \kappa \), and let \( \M \in \Mod^\kappa_T \). If \( [\M]_{\cong} \) has nonempty interior (with respect to \( \Mod^\kappa_T \)), then it is dense in \( \Mod^\kappa_T \).
\end{prop}

\begin{proof}
Let \( \Q_0 \) be such that \( \emptyset\neq\Nbhd_{\Q_0} \cap \Mod^\kappa_T \subseteq [\M]_{\cong} \), and fix an 
arbitrary \( \Q_1 \) such that \( \Nbhd_{\Q_1} \cap \Mod^\kappa_T \neq \emptyset \) (in particular, \( \Q_0 \) and \( \Q_1 \) are suitable reducts of bounded substructures of \( \kappa \)-sized model of \( T \)): we want to show 
that there is \( \N_1 \cong \M \) such that \mbox{\( \N_1 \in \Nbhd_{\Q_1} \cap \Mod^\kappa_T \).} By applying Lemma \ref{lem:jep} to any \mbox{\(\M_0\in \Nbhd_{\Q_0} \cap \Mod^\kappa_T \)} and \( \M_1\in\Nbhd_{\Q_1} \cap \Mod^\kappa_T \) we obtain $\N\in\Mod^\kappa_T$ such that $\M_0$ and $\M_1$ (and therefore $\Q_0$ and $\Q_1$) embed into it. Appropriate permutations of $\N$ will induce isomorphic copies \mbox{$\N_0\in \Nbhd_{\Q_0} \cap \Mod^\kappa_T$} and $\N_1\in \Nbhd_{\Q_1} \cap \Mod^\kappa_T$. By our choice of $\Q_0$ we have \mbox{$\N_0\cong\M$}, whence also $\N_1\cong\N\cong\N_0\cong\M$.
\end{proof}

As a corollary, we get the following purely topological characterization of categoricity.

\begin{theorem} \label{thm:catcharbounded}
Let \( \kappa \) be any infinite cardinal and \( T \) be a (not necessarily countable) complete first-order theory in a language \( \L \) of size \( \leq \kappa \). The following are equivalent:
\begin{enumerate-(1)}
\item \label{thm:catcharbounded-1}
\( T \) is \( \kappa \)-categorical;
\item \label{thm:catcharbounded-2}
\( \mathrm{rk}_B({\cong^\kappa_T}) = 0 \), i.e.\ \( \cong^\kappa_T \) is clopen;
\item \label{thm:catcharbounded-4}
\( [\M]_{\cong}\) is (cl)open for every \( \M \in \Mod^\kappa_T \);
\item \label{thm:catcharbounded-5}
there is \( \M \in \Mod^\kappa_T \) such that \( [\M]_{\cong} \) is clopen.
\end{enumerate-(1)}
In particular, there is no complete non-\( \kappa \)-categorical first-order theory \( T \) for which \( \cong^\kappa_T \) is a nontrivial open set.
\end{theorem}

\begin{proof}
The unique nontrivial implication is \ref{thm:catcharbounded-5} \( \IMP \) \ref{thm:catcharbounded-1}. Assume towards a contradiction that \( [\M]_{\cong} \) is clopen but \( T \) is not \( \kappa \)-categorical. Then \( [\M]_{\cong} \) would trivially have nonempty interior and \( \Mod^\kappa_T \setminus [\M]_{\cong} \) would be a nonempty open set, contradicting Proposition~\ref{prop:density}.

The additional part follows from the fact that if \( \cong^\kappa_T \) is open, then so is \( [\M]_T \) for all \( \M \in \Mod^\kappa_T \).
\end{proof}
 
In particular, by the Morley-Shelah theorem \cite[Theorem 12.2.1]{hodges1993model},
  \( \cong^\kappa_T \) is (cl)open for some uncountable \( \kappa>|T| \) if and only if the same happens for all uncountable $\kappa>|T|$. We will see in Section~\ref{subsec:someexamples} that this does not hold if \( \cong^\kappa_T \) is more complicated.

In the countable case the above results can be further improved. The following proposition shows that the assumption on \( \M \) in Proposition~\ref{prop:density} can be removed when \( \kappa  = \omega \), the reason being that in this case the bounded topology \( \tau_b \) coincides with the product topology (see also Remark~\ref{rmk:productcate}),

\begin{prop} \label{lem:density}
Let \( T \) be a complete first-order theory in a countable language \( \L \).  Then for every \mbox{\( \M \in \Mod^\omega_T \),} its isomorphism class \( [\M]_{\cong}  \) is dense in \( \Mod^\omega_T \). 
\end{prop}

\begin{proof}
Let \( \L' \subseteq \L \) be finite and \( \Q \) be any finite \( \L' \)-structure with domain contained in \( \omega \) such that \mbox{\( \Nbhd_{\Q} \cap \Mod^\omega_T \neq \emptyset \)} (i.e.\ \( \Q \) is a finite substructure of the \( \L' \)-reduct of some element of \( \Mod^\omega_T \)).
Let \( (\alpha_i)_{i < n} \) be the increasing enumeration of the domain of \( \Q \), let%
\footnote{Here we consider first-order logic with equality, thus we include e.g.\ the formulas \( \neg(y_i=y_j) \) for \( i \neq j \) when building such \( \uptheta \).}
 \( \uptheta((y_i)_{i <n}) \) be the (conjunction of the formulas in the) \( \L' \)-atomic diagram of \( \Q \),
 and let \(\uppsi \) be the first-order sentence
\[ 
\exists y_0 \dotsc \exists y_{n-1} \,   \uptheta((y_i)_{i < n}) .
 \] 
By our assumption on \( \Q \), the set \( T \cup \{ \uppsi \} \) is consistent, thus \( T \models \uppsi \) by completeness of \( T \). It follows that for every \( \M \in \Mod^\omega_T \), there are distinct \( \beta_0, \dotsc, \beta_{n-1}\in \omega \) such that \( \M \models \uptheta[\beta_0, \dotsc, \beta_{n-1}] \). Fix any permutation \( p \) of \( \omega \) such that \( p(\beta_i) = \alpha_i \) for every \( i < n \), and let \( \N \) be the (unique) \( \L \)-structure on \( \omega \) which is isomorphic to \( \M \) via \( p \): then \( \N \models T \) and
\( \N \in \Nbhd_\Q \cap \Mod^\omega_T \cap [\M]_{\cong} \). 
\end{proof}

\begin{cor} \label{cor:density}
Let \( T \) be a complete first-order theory in a countable language. Then there is no nontrivial open or closed set \( A \subseteq \Mod^\omega_T \) which is closed under isomorphism. 
\end{cor}

\begin{proof}
Assume towards a contradiction that there is such an \( A \subseteq \Mod^\omega_T \). Without loss of generality, we may assume that \( A \) is open (otherwise we replace \(A \) with \( \Mod^\omega_T \setminus A \)). Since \( A \neq \Mod^\omega_T \), there is \( \M \in \Mod^\omega_T \setminus A \), while since \( A \) is open and nonempty we get \( [\M]_{\cong} \cap A \neq \emptyset \) by Proposition~\ref{lem:density}, contradicting the fact that \( A\) is closed under isomorphism.
\end{proof}

The following theorem strengthens Theorem~\ref{thm:catcharbounded} in the case \( \kappa =\omega \). The unique nontrivial implication (namely, ~\ref{thm:catchar-5} $\IMP$~\ref{thm:catchar-1}) follows from Corollary~\ref{cor:density}.

\begin{theorem} \label{thm:catchar}
Let \( T \) be a complete first-order theory in a countable language. The following are equivalent:
\begin{enumerate-(1)}
\item \label{thm:catchar-1}
\( T \) is \(\omega\)-categorical;
\item \label{thm:catchar-2}
\( \cong^\omega_T \) is clopen;
\item \label{thm:catchar-3}
\( \mathrm{rk}_B({\cong^\omega_T}) \leq 1 \), i.e.\ \( \cong^\omega_T \) is open or closed;
\item \label{thm:catchar-4}
all isomorphism classes of the \( \omega \)-sized models of \( T \) are open or closed;
\item \label{thm:catchar-5}
there is \( \M \in \Mod^\omega_T \) such that \( [\M]_{\cong} \) is open or closed.
\end{enumerate-(1)}
In particular, there is no complete first-order theory such that \( \cong^\omega_T \) is a true open or a true closed set.
\end{theorem}

\begin{remark}
We currently do not know if there can be an \emph{uncountable} cardinal \( \kappa \) and a complete first-order theory \( T \) (in a language of size \( \leq \kappa \))  such that \( T \) is not \( \kappa \)-categorical, yet \( \cong^\kappa_T \) is a true closed set. However, in such case \( [\M]_{\cong} \) would be nowhere dense for every \( \M \in \Mod^\kappa_T \): indeed, \( [\M]_{\cong} \) would be closed by Fact~\ref{fct:fromertoclass} and \( \Mod^\kappa_T \setminus [\M]_{\cong} \) would be nonempty and open, hence it would be enough to apply Proposition~\ref{prop:density} to get the desired conclusion.
In particular, if \( \kappa = \aleph_\gamma \) is such that \( \kappa^{< \kappa} = \kappa  \) and \( \beth_{\omega_1} \left( |\gamma| \right) \leq \kappa \), then in the above scenario \( \Mod^\kappa_T \) would be \( \kappa \)-meager 
in itself (hence also in \( \Mod^\kappa_\L \)): this means that from the topological point of view such a \( T \) would have very few \( \kappa \)-sized models. 
\end{remark}

\begin{remark} \label{rmk:productcate}
Theorem~\ref{thm:catchar} can be extended to uncountable cardinals \( \kappa \) if \( \Mod^\kappa_T \) is endowed with the so-called \emph{product topology} rather than the logic one, where the product topology is the one generated by the sets
\[ 
\Nbhd_\Q = \{ \M \in \Mod^\kappa_\L \mid \Q \text{ is a substructure of the \( \L' \)-reduct of } \M  \},
 \] 
with \( \L' \subseteq \L \) finite and \( \Q \) a finite \( \L' \)-structure with domain contained in \( \kappa \). (Notice that with this topology \( \Mod^\kappa_\L \) becomes homeomorphic to \( \pre{\kappa}{2} \) when the latter is endowed with the product of the discrete topology on \( 2 \).)
\end{remark}

We end this section with a quite unexpected phenomenon uncovered by Theorem~\ref{thm:catchar} in the 
context of (classical) Borel reducibility, see~\cite{Gao2009} for a standard reference on this subject. Recall that an equivalence relation \( E \) on a topological 
(usually: Polish) space \( X \) is called \emph{smooth} if there is a Borel function \( f \colon X \to Y \) with 
\( Y \) Polish%
\footnote{Without loss of generality, since all uncountable Polish spaces are  Borel isomorphic to each other one can always 
take \( Y = \RR \).}
 such that for all \( x,y \in X \)
\[ 
x \mathrel{E} y \IFF f(x) = f(y).
 \] 
This basically means that the elements of  \( X \) can be classified up to \( E \)-equivalence 
using reals as complete invariants; since real numbers are well-understood, smooth equivalence 
relations are thus often called \emph{concretely classifiable}~\cite[Definition~5.4.1]{Gao2009}. However, in some situations such 
classification may fail to be completely satisfactory because the Borel classifying map \( f \) could be 
quite complicated, and thus it could be practically unfeasible to compute the invariant \( f(x) \) from a 
given input \( x \in X \). A natural and more adequate variation, naturally related to the notion of continuous reducibility considered e.g.\ in~\cite[Definition 5.1.2]{Gao2009} or~\cite{Thomas2009}, could be the following.

\begin{defin} \label{def:topsmooth}
An equivalence relation \( E \) on a topological space \( X \) is called \emph{topologically smooth} if there is a Hausdorff space \( Y \) and a continuous map \( f \colon X \to Y \)  such that for all \( x,y \in X \)
\[ 
x \mathrel{E} y \IFF f(x) = f(y).
 \] 
\end{defin}

Notice that the more stringent requirement that the classifying map \( f \) be continuous is balanced by the fact that we allow the classifying objects to form a more complicated topological space, if desired; indeed, being Hausdorff is arguably the minimal requirement in order to clearly distinguish two given invariants from each other (but obviously nothing prevents us to further require \( Y \) to be a nicer space, e.g.\ a Polish one).

It is not hard to see that there are theories \( T \) for which \( \cong^\omega_T \) is smooth: consider e.g.\  theories having at most countably many countable models. What about topological smoothness? Of course if \( T \) is \(\omega\)-categorical, then \( \cong^\omega_T \) is (trivially) topological smooth. Quite surprisingly, the next result shows that if instead \( T \) is  not \(\omega\)-categorical, then it is not possible to classify its countable models in a continuous way, even in the very simple case in which \( T \) has just finitely many models. 

\begin{theorem} \label{thm:topsmooth}
Let \( T \) be a complete first-order theory in a countable language. Then \( \cong^\omega_T \) is topologically smooth if and only if \( T \) is \(\omega\)-categorical.
\end{theorem}

\begin{proof}
For the nontrivial direction, suppose that \( f \colon X \to Y \) witnesses that \( T \) is topologically smooth. Then \( \cong^\omega_T \) is the preimage under the product function \( f \times f \) of the diagonal of \( Y \). Since \( Y \) is Hausdoff, the latter is closed, and hence so is \( \cong^\omega_T \) because \( f \times f \) is continuous. By Theorem~\ref{thm:catchar} we can then conclude that \( T \) is \(\omega\)-categorical.
\end{proof}

\begin{remark}
By Remark~\ref{rmk:productcate}, similar considerations apply to the isomorphism relations \( \cong^\kappa_T \) with \( \kappa > \omega \), provided that the continuity requirement in Definition~\ref{def:topsmooth} is referred to the product topology on \( \Mod^\kappa_T \).
\end{remark}

In particular, Theorem~\ref{thm:topsmooth} provides another topological characterization of \mbox{\(\omega\)-categoricity,} while from the point of view of Borel reducibility it  yields many natural examples of pairs of analytic equivalence relations such that one is Borel reducible yet not continuously reducible to the other one, a notoriously difficult problem in descriptive set theory (see the discussion in the introduction of~\cite{Thomas2009}). It is maybe worth noticing that our examples are even smooth, thus simpler than previously known instances of this phenomenon, and can be chosen so that they have countably many (or even finitely many) equivalence classes.

\section{Scott height and Borel rank} \label{sec:ScottVsBorel}

\subsection{Scott height and Ehrenfeucht-Fra\"{i}ss\'e games}\label{inf}

Given two \( \kappa \)-sized \( \L \)-structures \( \M , \N \) and  \( \alpha\in \mathrm{On} \), we write $\M\equiv_\alpha\N$ if $\M$ and $\N$ satisfy the same $\L_{\infty\kappa}$-sentences of quantifier rank $\leq\alpha$. 

\begin{defin}
Let \( \ElCl \subseteq \Mod^\kappa_\L \) be an invariant set and \( \M \in \ElCl \). The $\L_{\infty\kappa}$-\textit{Scott height}%
\footnote{This is slightly different from the definition of Scott height found in \cite{FHK14}.}
of $\M$ (with respect to \( \ElCl \)) is
\[
S(\kappa,\ElCl,\M)=\min\left\{\alpha \in \mathrm{On} \mid\forall\N\in \ElCl\ (\M\equiv_\alpha\N\IMP \M\cong\N)\right\}.
\]
If there is no such ordinal, we set $S(\kappa,\ElCl,\M)=\infty$. 
	
When \( \ElCl \) is axiomatized by a first-order theory \( T \) we write \( S(\kappa,T, \M) \) instead of \( S(\kappa, \Mod^\kappa_T, \M) \), and similarly when \( \ElCl \) is axiomatized by an \( \L_{\kappa^+ \kappa} \)-sentence \( \upvarphi \).
\end{defin}

For our purposes, the following oft-overlooked notion becomes crucial.

\begin{defin}
The $\L_{\infty\kappa}$-\textit{Scott height} of an invariant set $\ElCl \subseteq \Mod^\kappa_\L $ is the supremum of the $\L_{\infty\kappa}$-Scott heights of its $\kappa$-sized models, i.e.
\[
S(\kappa,\ElCl)=\sup\{S(\kappa,\ElCl,\M)\mid \M\in \ElCl \}.
\]
We again simplify the notation writing \( S(\kappa,T) \) and \( S(\kappa, \upvarphi) \) rather than \( S(\kappa,\Mod^\kappa_T) \) and \( S(\kappa,\Mod^\kappa_\upvarphi) \), respectively.
\end{defin}

\begin{remark} \label{rmk:categorical}
Recall that when \( \L \) is a relational language, then \( \M \equiv_0 \N \) for all \( \M,\N \in \Mod^\kappa_\L \) because there are no \( \L_{\infty \kappa } \)-sentences with quantifier rank \( 0 \). Thus the following are equivalent for any \( \L \)-theory \( T \):
\begin{enumerate-(a)}
\item
\( T \) is uncountably categorical;
\item
for some/any \( \kappa > \omega \) there is \( \M \in \Mod^\kappa_T \) with \( S(\kappa,T,\M) = 0 \);
\item
for some/any \( \kappa > \omega \), \( S(\kappa,T) = 0 \).
\end{enumerate-(a)}
More generally, if \( \ElCl \subseteq \Mod^\kappa_\L  \) is an invariant set we have: \( \ElCl \) consists of a single isomorphism class if and only if  there is \( \M \in \ElCl \) with \( S(\kappa,T,\M) = 0 \), if and only if \(  S(\kappa,\ElCl) = 0 \).
\end{remark}

A useful way to deal with these notions is via Ehrenfeucht-Fraiss\'e games.
\begin{defin}
	Let $\Tr$ be a well-founded $\leq\kappa$-sized tree and let $\M$ and $\N$ be models (with domain $\kappa$).   In the \textit{Ehrenfeucht-Fra\"{\i}ss\'e game} $\EF_\Tr^\kappa(\M,\N)$, at every step player $\1$ plays a pair \( (p,C) \) where \( p \) is a node of $\Tr$ and \( C \) is a subset of $\kappa$, while player $\2$ picks a partial function $f \colon \kappa\to\kappa$. The rules are as follows. Suppose the sequence of moves $((p_i,C_i),f_i)_{i<n}$ has been played. Then:
	\begin{itemize}
		\item player $\1$ picks a node $p_n\in \Tr$ which is an immediate successor of $p_{n-1}$ (or of the root, if $n=0$) and a subset $C_n\subset\kappa$ of size less than $\kappa$ such that $C_n\supseteq C_i$ for every $i<n$;
		\item player $\2$ picks a partial function $f_n \colon \kappa\imp\kappa$ such that $|\dom(f_n)|<\kappa$, $\dom(f_n)\cap\ran(f_n)\supseteq C_n$,  and $f_n\supseteq f_i$ for every $i<n$.
	\end{itemize}
	The game ends when player $\1$ runs out of nodes to pick from, that is, she cannot move on the \( n \)-th round because \( p_{n-1} \) is a leaf of \( \Tr \) (this must happen at some stage \( n < \omega \) because \( \Tr \) is well-founded).
	Then player $\2$ wins if $f=\bigcup_{i < n} f_i$ is a partial isomorphism between \( \M \) and \( \N \), otherwise player $\1$ wins.
	We write $\Omega\wins\EF_\Tr^\kappa(\M,\N)$ to indicate that player $\Omega$ has a winning strategy in \( \EF_\Tr^\kappa(\M,\N) \), and $\Omega\not\wins\EF_\Tr^\kappa(\M,\N)$ to indicate she has none. 
\end{defin}

\begin{remark} \label{rmk:specialcase}
When \( \Tr \) has rank \( 1 \), i.e.\ when it consists only of its root, there is no possible first move for \( \1 \): this means that \( f = \bigcup_{i < n} f_i \) is the empty function, which is always a partial isomorphism. We conclude that for such a \( \Tr \) we have \( \2 \wins \EF_\Tr^\kappa(\M,\N) \) for any \( \M,\N \).
\end{remark}

\begin{remark} \label{rmk:run}
The possible runs of  $\EF_\Tr^\kappa(\M,\N)$ are independent of both \( \M \) and \( \N \), which are instead involved only in the winning condition of the game. For this reason we will speak about ``run(s) in the game $\EF_\Tr^\kappa(\cdot,\cdot)$'' when we want to refer to run(s) in some/any game of the form $\EF_\Tr^\kappa(\M,\N)$ for \( \M,\N \in \Mod^\kappa_\L \).
\end{remark}

To simplify the notation, we write $\EF_\alpha^\kappa(\M,\N)$ for $\EF_{\Tr_\alpha}^\kappa(\M,\N)$.
The link between EF-games and Scott height is the following result.

\begin{theorem}\label{qr}~\cite[Theorem 9.27]{Va11}
For any two models $\M$ and $\N$ and every ordinal \(\alpha\),
\[
\M\equiv_\alpha\N\IFF\ \2\wins\EF_{\alpha}^\kappa(\M,\N).
\]
\end{theorem}

 It follows that
\[
S(\kappa,\ElCl,\M)=\min\left\{\alpha\in \mathrm{On}\mid\forall\N\in \ElCl\ (\2\wins\EF_{\alpha}^\kappa(\M,\N)\IMP \M\cong\N)\right\}.
\]

\subsection{Scott height and Borel rank}

The following two theorems refine the two directions of~\cite[Theorem 65]{FHK14} in order to obtain an explicit connection between Scott height of a theory and Borel rank of its isomorphism relation. We also consider arbitrary invariant sets \( \ElCl \) rather than just those of the form \( \Mod^\kappa_T \) for \( T \) a first-order theory: this will allow us to obtain further corollaries concerning e.g.\ \( \L_{\kappa^+ \kappa} \)-elementary classes. The proofs are essentially unchanged: the novelty is the explicit computation of the ranks involved.

Let us fix an invariant set
$\ElCl \subseteq \Mod^\kappa_\L$, with \( \L \) a relational language as in Section~\ref{subsec:codes}. We consider the isomorphism relation $\cong_{\ElCl}^\kappa$ between structures in \( \ElCl \) as a subset of $\ElCl \times \ElCl$. In particular, \( \mathrm{rk}_B({\cong}^\kappa_{\ElCl}) \) is computed relatively to $\ElCl \times \ElCl$, that is, \( \mathrm{rk}_B({\cong^\kappa_{\ElCl}}) \) is the smallest \( \delta \in \mathrm{On} \) such that
\[
\left\{(\M,\N)\in \ElCl^2\mid  \M\cong\N\right\}=D\cap\ElCl^2
\]
for some $D\in\boldsymbol{\Sigma}_\delta^0\left((\mathrm{Mod}^\kappa_\L)^2\right) \cup \boldsymbol{\Pi}_\delta^0\left((\mathrm{Mod}^\kappa_\L)^2\right)$ if such a \(\delta\) exists, and \( \mathrm{rk}_B({\cong^\kappa_{\ElCl}}) = \infty \) otherwise.

\begin{theorem}\label{link1}
	Let $\kappa^{<\kappa}=\kappa$. If $\mathrm{rk}_B({\cong^\kappa_{\ElCl}}) \neq \infty$, then for every $\M,\N\in \ElCl$ 
\[
\M\cong\N\IFF \2\wins\EF_{\max\{ \mathrm{rk}_B({\cong^\kappa_{\ElCl}}),1 \}}^\kappa(\M,\N).
\]
\end{theorem}

\begin{proof}
First, we extend  our language $\L = \{ R_i \mid i < I \}$ (where each \( R_i \) is of arity \( n_i \)) to the language \( \bar \L = \L \cup \{ P \} \), where \( P \) is a new unary relational symbol. 
The first step is to turn $\cong_{\ElCl}^\kappa$ into a Borel subset $A$ (with same Borel rank) of a suitable invariant set $\ElCl' \subseteq \Mod^\kappa_{\bar\L}$, so that we can apply Theorem~\ref{lopezext}.

Let 
\[
W=\left\{\bar\A\in \Mod^\kappa_{\bar\L}\mid \left|P^{\bar\A}\right|=\left|\kappa\setminus P^{\bar\A}\right|=\kappa\right\}.
\]
Clearly \( W \) is closed under isomorphism, and for every $\bar\A\in W$ there are unique order preserving bijections
\[
\tau^{\bar\A}_1 \colon \kappa\to P^{\bar\A} \quad \text{and} \quad \tau^{\bar\A}_2 \colon \kappa\to\kappa\setminus P^{\bar\A}.
\]
We can then define a continuous surjective map $h \colon W\to(\Mod^\kappa_{\L})^2$ such that 
\[
h(\bar\A)=(h_1(\bar\A),h_2(\bar\A))=(\A_1,\A_2)
\]
where for \( k = 1,2 \) we let  $\A_k$ be the \( \L \)-structure defined by setting
\[
R_i^{\A_k}(x_1,\dots,x_{n_i})\IFF R_i^{\bar{\A}}(\tau^{\bar{\A}}_k(x_1),\dots,\tau^{\bar{\A}}_k(x_{n_i}))
\]
for every $i<I$. Let 
\[
\ElCl'=W\cap h^{-1}\left[\ElCl \times \ElCl\right].
\]
Our hypothesis is that 
\( \mathrm{rk}_B({\cong^\kappa_{\ElCl}})  < \kappa^+ \), hence also \( \delta = \max \{ \mathrm{rk}_B({\cong^\kappa_{\ElCl}}), 1 \} < \kappa^+ \).
By continuity of $h$ we have that
\[
A=\left\{{\bar{\A}}\in \ElCl' \mid h_1({\bar{\A}})\cong h_2({\bar{\A}})\right\}=h^{-1}({\cong^\kappa_{\ElCl}})
\]
is such that \( \mathrm{rk}_B(A) \leq \mathrm{rk}_B({\cong^\kappa_{\ElCl}}) \),
and clearly both $\ElCl'$ and $A$ are closed under isomorphisms because ${\bar{\A}}\cong{\bar{\B}}$ implies both  ${\bar{\A}}\cap P^{{\bar{\A}}}\cong{\bar{\B}}\cap P^{{\bar{\B}}}$ and $ {\bar{\A}}\setminus P^{{\bar{\A}}}\cong {\bar{\B}}\setminus P^{{\bar{\B}}}$: thus by Theorem~\ref{lopezext} there exists an \( \bar{\L}_{\kappa^+\kappa} \)-sentence $\upsigma$ with $R(\upsigma) \leq \max \{ \mathrm{rk}_B(A),1 \} \leq \delta$  such that 
\[
A=\left\{{\bar{\A}}\in \ElCl'\mid {\bar{\A}}\models \upsigma\right\}.
\]
It follows that \( \bar{\A}\not\equiv_\delta\bar{\B}\) for all ${\bar{\A}} \in A \) and \( \bar{\B}\in \ElCl' \setminus A$. 
By Theorem \ref{qr}, this means that for all ${\bar{\A}},\bar{\B}\in \ElCl'$ 

\begin{equation}\label{game}
\text{if } {\bar{\A}}\in A \text{ and }\ {\bar{\B}}\notin A, \text{ then }\2\not\wins\EF^\kappa_{{\delta}}({\bar{\A}},{\bar{\B}}).
\end{equation}

We use this fact to show that for all $\M,\N\in \ElCl$
\[
\M\cong\N\IFF\2\wins\EF^\kappa_{{\delta}}(\M,\N),
\]
as desired.
The direction ($\IMP$) is obvious by definition of $\EF$-games. In order to prove $(\Leftarrow)$, suppose towards a contradiction that there are nonisomorphic models $\M,\N \in \ElCl$ such that $\2\wins\EF^\kappa_{{\delta}}(\M,\N)$. We define ${\bar{\A}},{\bar{\B}}\in X$ by setting:
\begin{itemize}
\item 
$P^{\bar{\A}}=P^{\bar{\B}}=\{2\alpha\mid\alpha<\kappa\}$, so that  
\[
\tau_1^{\bar{\A}}=\tau_1^{\bar{\B}} \colon \kappa\to P^{\bar{\A}} \colon \alpha\mapsto2\alpha
\] 
and 
\[ 
\tau_2^{\bar{\A}}=\tau_2^{\bar{\B}} \colon\kappa\to\kappa\setminus P^{\bar{\A}} \colon\alpha\mapsto2\alpha+1,
\]
\end{itemize}
and  for every $R_i \in \L$  
\begin{itemize}
\item
\( R^{\bar{\A}}(\tau^{\bar{\A}}_1(x_1),\dots,\tau^{\bar{\A}}_1(x_{n_i})) \IFF R^{\bar{\A}}(\tau^{\bar{\A}}_2(x_1),\dots,\tau^{\bar{\A}}_2(x_{n_i})) \IFF R^{\M}(x_1,\dots,x_{n_i})\) 
\item
\( R^{\bar{\B}}(\tau^{\bar{\B}}_1(x_1),\dots,\tau^{\bar{\B}}_1(x_{n_i})) \IFF  R^{\M}(x_1,\dots,x_{n_i})\), while  \( R^{\bar{\B}}(\tau^{\bar{\B}}_2(x_1),\dots,\tau^{\bar{\B}}_2(x_{n_i}))\IFF R^{\N}(x_1,\dots,x_{n_i}) \)
\item 
for every \( x_1,\dotsc, x_{n_i} \in \kappa \), if there are \( 1 \leq k,k' \leq n_i \) such that \( P^{\bar{\A}}(x_k) \) and \( \neg P^{\bar{\A}}(x_{k'}) \) then \( \neg R^{\bar{\A}}(x_1, \dotsc, x_{n_i}) \), and the same with \( \bar{\A} \) replaced by \( \bar{\B} \).
\end{itemize}  
Note that $h({\bar{\A}})=(\M,\M)$ and $h({\bar{\B}})=(\M,\N)$.

\begin{claim} \label{claim:2wins}
\( \2 \wins \EF^\kappa_{{\delta}}({\bar{\A}},{\bar{\B}}) \).
\end{claim}	

\begin{proof}[Proof of the claim] 
 Clearly $\2\wins\EF^\kappa_{{\delta}}(\M,\M)$, and we assumed 
 $\2\wins\EF^\kappa_{{\delta}}(\M,\N)$. Let $\sigma_1$ and $\sigma_2$ be two winning strategies for \( \2 \) 
 in the respective games. We are going to show how to combine $\sigma_1$ and $\sigma_2$ in order to 
obtain a winning strategy for $\2$ in $\EF^\kappa_{{\delta}}({\bar{\A}},{\bar{\B}})$. To this aim, first 
note that any partial isomorphisms $f_1 \colon \M\to\M$ and $ f_2\colon\M\to\N$ induce partial 
isomorphisms $\overline {f_1}\colon{\bar{\A}}\cap P^{{\bar{\A}}}\to{\bar{\B}}\cap P^{{\bar{\B}}}$ and \mbox{$\overline{f_2}\colon{\bar{\A}}\setminus P^{{\bar{\A}}}\imp {\bar{\B}}\setminus P^{{\bar{\B}}}$,} 
whose union is still a partial isomorphism $f\colon{\bar{\A}}\to {\bar{\B}}$ (because of the last condition 
in the definition of ${\bar{\A}}$ and ${\bar{\B}}$). Suppose player $\1$ has last played $(p,C)$  in $\EF^\kappa_{{\delta}}({\bar{\A}},{\bar{\B}})$ for some $p\in \Tr$ and $C\subset\kappa$. Let 
\[ 
C_1=\{\alpha<\kappa\mid 2\alpha\in C\} \quad \text{and} \quad \ C_2=\{\alpha<\kappa\mid 2\alpha+1\in C\}.
\]
Then the map
\[
(p,C)\mapsto\overline{ \sigma_1((p,C_1))}\cup\overline{\sigma_2((p,C_2))}
\] 
is clearly a winning strategy for player $\2$.
\end{proof}

Now notice that ${\bar{\A}}\in A$ because \(h({\bar{\A}})=(\M,\M) \), while ${\bar{\B}}\notin A$ because \mbox{\( h({\bar{\B}})=(\M,\N) \)} and we assumed $\M\ncong\N$. Thus $\2\not\wins\EF^\kappa_{{\delta}}({\bar{\A}},{\bar{\B}})$ by (\ref{game}), contradicting Claim~\ref{claim:2wins}.
\end{proof}

\begin{theorem}\label{link2}
Let $\kappa^{<\kappa}=\kappa$ and assume that \( | \L | < \kappa \). Suppose that there is  a well-founded $\leq\kappa$-sized tree $\Tr$ with  $\varrho(\Tr) = \beta+1 $ such that  for every $\M,\N\in \ElCl$
\[
\M\cong\N\IFF \2\wins\EF_\Tr^\kappa(\M,\N).
\]
Then ${\cong_\ElCl^\kappa} \in \boldsymbol{\Pi}_{2\beta}^0(\ElCl^2)$, whence, in particular, \( \mathrm{rk}_B({\cong^\kappa_\ElCl}) \leq  2\beta \).
\end{theorem}

\begin{proof}
Let $\mathscr{U}$ be the tree of all partial runs in $\EF_\Tr^\kappa(\cdot,\cdot)$, that is, the tree generated by the sequences of the form 
\[
\langle(p_0,C_0),f_0,\dots,(p_{n-1},C_{n-1}),f_{n-1}\rangle
\]
such that:
\begin{itemize}
\item 
$p_0$ is an immediate successor of the root $r$ of $\Tr$, and $p_i$ is an immediate successor of $p_{i-1}$ for every $0 < i < n$;
\item 
$C_i\subseteq C_j\subset\kappa$ with $ |C_j|<\kappa$ for every $0 \leq i\leq j< n$;
\item 
$f_j \colon \kappa\to\kappa$ is a partial function such that $\dom(f_j)\cap\ran(f_j)\supseteq C_j$, \mbox{$|\dom(f_j)|<\kappa$} and $f_i\subseteq f_j$ for every $0 \leq i\leq j< n$.
\end{itemize}

Clearly $\mathscr{U}$ is still a well-founded tree because every branch of $\mathscr{U}$ is long twice some branch of $\Tr$ (plus the root), and it is of size \( \leq \kappa \) because the amount of successors of any node of \( \mathscr{U} \), which is determined by the number of possible moves of \( \1 \) and \( \2 \) in a round of $\EF_\Tr^\kappa(\cdot,\cdot)$, is at most $\kappa^{<\kappa}=\kappa$.  Notice also that the leaves of \( \mathscr{U} \) are exactly the maximal runs in \( \EF_\Tr^\kappa(\cdot,\cdot)$, that is sequences
\[ 
\langle(p_0,C_0),f_0,\dots,(p_{n-1},C_{n-1}),f_{n-1}\rangle
 \] 
where \( p_{n-1} \) is a leaf of \( \Tr \). (When \( \beta = 0 \)  we have \( \mathscr{U} = \{ \emptyset \} \), and the unique maximal branch of \( \mathscr{U} \) is \( \langle \emptyset \rangle \).)

\begin{claim} \label{rankU}
$\varrho(\mathscr{U}) = 2\beta+1$.
\end{claim}

\begin{proof}[Proof of the claim]
By induction on $\beta$. The case $\beta=0$ is clear. Suppose $\beta>0$. Then the root $r$ of $\Tr$ has rank $\beta$, and its immediate successors $\{q_i\mid i<\kappa\}$ have rank $\beta_i<\beta$.  By the inductive hypothesis every node of the form $\langle (q_i,C),f\rangle\in\nolinebreak \mathscr{U}$ has rank $2\beta_i$, thus every node of the form $\langle(q_i,C)\rangle$ has rank $2\beta_i+1$ and the root has rank $\sup\{2\beta_i+2\mid i<\kappa\}=2\beta$, whence $\mathscr{U}$ has rank $2\beta+1$.	\end{proof} 
	
Next we define a labeling function
$\ell$ on the leaves  
\[
b=\langle(p_0,C_0),f_0,\dots,(p_{n-1},C_{n-1}),f_{n-1}\rangle
\]
of \( \mathscr{U} \) by setting \( f = \bigcup_{i < n} f_i = f_{n-1} \) and

\begin{equation} \label{eq:ellb}
\ell(b)=\left\{(\M,\N)\in \ElCl^2 \mid f \colon \M\imp\N \text{ is a  partial isomorphism}\right\}.
\end{equation}

(Notice that when \( \beta = 0 \) then necessarily \( n = 0 \), so that \( f \) is the empty function and \( \ell(b) = \ElCl^2 \).)
We claim%
\footnote{Here is where we use the assumption \( | \L | < \kappa \), which ensures that we can take \( \Q \) and \( \R \) below to be substructures of \( \M \) and \( \N \), and not just of suitable reducts of them.}
 that $\ell(b)$ is a (relatively) clopen subset of \( \ElCl^2 \). 
Let \( (\M,\N) \in \ell(b) \), and let
 \( \Q \) and \( \R \) be the substructures of, respectively, \( \M \) and \( \N \) with domain  $\dom (f)\cup\ran (f)$. Since \( | \dom (f)\cup\ran (f) | < \kappa \) and \( \kappa \) is regular, \( \Nbhd_\Q \) and \( \Nbhd_\R \) are basic clopen sets of \( \Mod^\kappa_\L \) and clearly
\[ 
(\M,\N) \in (\Nbhd_\Q \times \Nbhd_\R) \cap \ElCl^2 \subseteq \ell(b).
 \] 
This shows that \( \ell(b) \) is (relatively) open in $\ElCl^2$, and the same argument (applied to any \( (\M,\N) \notin \ell(b) \)) shows that it is also (relatively) closed.

Thus $(\mathscr{U},\ell)$ is a \( \kappa^+ \)-Borel code for a subset of $\ElCl^2$: we claim that 
$B(\mathscr{U},\ell)$ is exactly ${\cong_{\ElCl}^\kappa}$, that is that for every \( \M,\N \in \ElCl \)
\[
\M\cong \N\IFF\2\wins G(\mathscr{U},\ell,(\M,\N)).
\]
Indeed, by the hypothesis of the theorem it suffices to show
\[
\2\wins\EF_{\Tr}^\kappa(\M,\N)\IFF\2\wins G(\mathscr{U},\ell,(\M,\N)):
\]
but this follows immediately from the fact that the two games have essentially the same moves (by definition of $\mathscr{U}$), and the winning conditions for player $\2$ in the two games are equivalent (by definition of $\ell$). 

Since $\varrho(\mathscr{U})=2\beta+1$  by Claim~\ref{rankU}, by Theorem~\ref{orig1} we then get that \mbox{${\cong_{\ElCl}^\kappa} = B(\mathscr{U}, \ell)$} belongs to $\boldsymbol{\Pi}_{2\beta}^0(\ElCl^2)$, as required.
\end{proof}

We are now ready to prove our main technical result, which refines~\cite[Theorem 65]{FHK14} and yields to Theorem~\ref{thm:ranksintro} by setting \(\ElCl = \Mod^\kappa_T \) for \( T \) the first-order theory under consideration. 

\begin{defin} \label{def:rk}
Given an invariant set \( \ElCl \subseteq \Mod^\kappa_T \), we set
\[ 
B(\kappa,\ElCl) = \mathrm{rk}_B(\cong^\kappa_{\ElCl}).
 \] 
 When \( \ElCl \) is of the form \( \Mod^\kappa_T \) for some first-order theory \( T \) (respectively, of the form \( \Mod^\kappa_\upvarphi \) for some \( \L_{\kappa^+ \kappa} \)-sentence \( \upvarphi \)) we simply write 
 \( B(\kappa,T) \) (respectively, \( B(\kappa,\upvarphi) \)) rather than \( B(\kappa,\Mod^\kappa_T) \) (respectively, \( B(\kappa,\Mod^\kappa_\upvarphi) \)).
\end{defin}

\begin{theorem} \label{orig3}
Let $\kappa^{<\kappa}=\kappa$, assume that \( | \L | < \kappa \), and let \( \ElCl \subseteq \Mod^\kappa_\L \) be an invariant set.
\begin{enumerate-(1)}
\item \label{orig3-1}
$S(\kappa,\ElCl)\leq \max \{B(\kappa,\ElCl),1 \}$.
In particular, if \( \cong^\kappa_{\ElCl} \) is \( \kappa^+ \)-Borel, then \( S(\kappa,{\ElCl}) < \kappa^+ \). If moreover \( \ElCl = \Mod^\kappa_T \) for some complete first-order theory \( T \), then we further have \( S(\kappa,\ElCl) \leq B(\kappa, \ElCl) \).
\item \label{orig3-2}
If $S(\kappa,{\ElCl}) < \kappa^+$, then \( {\cong^\kappa_{\ElCl}} \in \boldsymbol{\Pi}^0_{2S(\kappa,\ElCl)}(\ElCl^2) \), whence \( \cong^\kappa_{\ElCl} \) is \( \kappa^+ \)-Borel. In particular, \( B(\kappa,\ElCl) \leq 2 S(\kappa,\ElCl) \).
\end{enumerate-(1)}
\end{theorem}

\begin{proof}
If \( B(\kappa,\ElCl) \neq \infty \), then by Theorem~\ref{link1} for every $\M,\N \in \ElCl$ we have \linebreak \mbox{$\2\wins\EF^\kappa_{\max \{ B(\kappa,\ElCl), 1 \}}(\M,\N)\IFF\M\cong\N$}, which means 
$S(\kappa,\ElCl)\leq \max \{ B(\kappa,\ElCl) ,1  \}$. The additional part when \( \ElCl \) is axiomatized by a complete first-order theory \( T \) follows from the fact that if \( \mathrm{rk}_B({\cong^\kappa_T}) = 0 \) then \( T \) is \( \kappa \)-categorical by Theorem~\ref{thm:catcharbounded}, hence also \( S(\kappa,\ElCl) = S(\kappa,T) = 0 \).
	
On the other hand, if \( S(\kappa,\ElCl) < \kappa^+ \) then $\Tr_{S(\kappa,\ElCl)}$, which is of rank \( S(\kappa,\ElCl)+1 \), witnesses the hypothesis of Theorem~\ref{link2}, so \( {\cong^\kappa_{\ElCl}} \in \boldsymbol{\Pi}^0_{2S(\kappa,\ElCl)}(\ElCl^2) \).
\end{proof}

\begin{cor} \label{cor:orig3}
Let $\kappa^{<\kappa}=\kappa$, assume that \( |\L|< \kappa \), and let \( \ElCl \subseteq \Mod^\kappa_\L \) be an invariant set.
\begin{enumerate-(1)}
\item \label{cor:orig3-1}
If one of \( B(\kappa,\ElCl) \) and \( S(\kappa,\ElCl) \) is \( < \kappa^+ \), then both of them are \( < \kappa^+ \) and
\[ 
S(\kappa,\ElCl)\leq \max \{ B(\kappa,\ElCl) , 1 \} \qquad \text{and} \qquad B(\kappa,\ElCl)\leq 2 S(\kappa,\ElCl),
 \] 
so that $S(\kappa,\ElCl)$ and $B(\kappa,\ElCl)$ have finite distance.
\item \label{cor:orig3-1.5}
If moreover \( \ElCl = \Mod^\kappa_T \) for a complete first-order theory \( T \), then actually
\[ 
S(\kappa,T) \leq B(\kappa,T) \leq 2 S(\kappa,T).
 \] 
\item \label{cor:orig3-2}
If one of \( B(\kappa,\ElCl) \) and \( S(\kappa,\ElCl) \) is a limit ordinal $<\kappa^+$, then 
\[
B(\kappa,\ElCl) = S(\kappa,\ElCl).
\]
\end{enumerate-(1)}
\end{cor}

\begin{proof}
For part~\ref{cor:orig3-2}, notice that if \( S(\kappa,\ElCl) \) is limit then \(  S(\kappa,\ElCl) = 2 S(\kappa,\ElCl) \), whence \( B(\kappa,\ElCl) = S(\kappa,\ElCl) \) by~\ref{cor:orig3-1}. Therefore it is enough to show that if \( S(\kappa,\ElCl) \) is not limit, then neither is \( B(\kappa,\ElCl) \): this follows again from~\ref{cor:orig3-1}, noticing that in such case all ordinals between \( S(\kappa,\ElCl) \) and \( 2 S(\kappa,\ElCl) \) are not limit, and either  \( B(\kappa,\ElCl)\) is among them or it is $0$.
\end{proof}

We will see in Section~\ref{subsec:someexamples} that part~\ref{cor:orig3-2} above may fail for successor ordinals even when considering the special case of invariant sets axiomatized by a countable complete first-order theory \( T \).

\begin{remark}
In some of the above results we had to require \( |\L| < \kappa \). This has no influence on the main results of the paper contained in Section~\ref{sec:descriptivemaingap}, as there we will be dealing with uncountable \( \kappa\)'s and countable languages \( \L \). As for the results of this section, notice that the hypothesis that \( \L \) be small is only used in the proof  of Theorem~\ref{link2} to ensure that the labelling function \( \ell \) defined after Claim~\ref{rankU} takes clopen sets as values. When \( \L \) is of size \( \kappa \), the set \( \ell(b) \) in~\eqref{eq:ellb} turns out to be in general closed: this just causes a minor modification to the indexes one gets. For example, in this situation we can still conclude at least that \( {\cong^\kappa_\ElCl} \in \boldsymbol{\Pi}^0_{2\beta+1}(\ElCl^2) \), so that \( \mathrm{rk}_B({\cong^\kappa_\ElCl}) \leq 2\beta +1 \). With this in mind, one can easily modify the subsequent results accordingly and get e.g.\ that \( S(\kappa,\ElCl) \) and \( B(\kappa, \ElCl) \) have finite distance also when \( |\L| = \kappa \).
\end{remark}

\section{A descriptive analogue to Shelah's Main Gap Theorem} \label{sec:descriptivemaingap}

\emph{For the rest of this section, we fix a countable complete first-order theory \( T \)} (which in particular means that the underlying language \( \L \) is countable as well).
Under certain cardinality hypotheses, Shelah's classification can be translated in terms of Scott height. We sum up in the following theorem the results we need.

\begin{theorem}[Shelah]
\begin{enumerate-(1)}
\item
Let $\kappa>2^{\aleph_0}$. If $T$ is  classifiable shallow of depth $\alpha$, then all of its $\kappa$-sized models have $\L_{\infty\kappa}$-Scott height $\leq2\alpha$.
\item
Let $\kappa\geq2^{\aleph_0}$. If $T$ is  superstable deep (in particular, if it is classifiable deep), then there are $\kappa$-sized models of $T$ with arbitrarily large $\L_{\infty\kappa}$-Scott height below $\kappa^+$.
\item
Let $\kappa>\omega$ regular. If $T$ is not classifiable, then there are $\kappa$-sized models of $T$ with $\L_{\infty\kappa}$-Scott height $\infty$. The converse holds as well if \( \kappa > 2^{\aleph_0} \).
\end{enumerate-(1)}
\end{theorem}

\begin{proof}
\begin{enumerate-(1)}
\item
This is a special case of  \cite[Theorem XIII.1.5]{Sh00}, the hypotheses of which come true because all countable theories with NOTOP have the existence property by  \cite[Conclusion XII.5.14]{Sh00}.
\item
This is a consequence of  \cite[Theorem XIII.1.8]{Sh00} obtained by letting $\mu=\lambda$.
\item
The first implication is \cite[Main Conclusion 0.2]{Sh87} when $T=T_1$, while the other direction  holds for every $\kappa>2^{\aleph_0}$ by  \cite[Theorem XIII.1.1]{Sh00}. \qedhere
\end{enumerate-(1)}
\end{proof}

We also need the following result by Lascar.

\begin{theorem}[Lascar, {\cite[Th\'eor\`eme 4.1]{La85}}]\label{depth}
If $T$ is classifiable shallow, the depth of $T$ is countable.
\end{theorem}

Summing up the above results,  we get the following corollary.
\begin{cor}\label{scottdepth}
Let $\kappa>2^{\aleph_0}$ be regular. Then $S(\kappa,T)<\infty$ if and only if $T$ is classifiable. Furthermore:
\begin{itemize}
\item 
if $T$ is classifiable shallow of depth $\alpha$, then $S(\kappa,T)\leq2\alpha<\omega_1$;
\item 
if $T$ is classifiable deep, then $S(\kappa,T)=\kappa^+$.
\end{itemize}
\end{cor}

Combining Corollary~\ref{scottdepth} with Theorem~\ref{orig3} we immediately obtain the following result, yielding in particular to Theorem~\ref{thm:main} and refining Theorem~\ref{thm:FHK} (that is, \cite[Theorem 63]{FHK14}).

\begin{theorem}\label{thm:DMGT}
Let $\kappa^{<\kappa}=\kappa>2^{\aleph_0}$. Then
\begin{itemize}
\item 
if $T$ is classifiable shallow of depth $\alpha$, then ${\cong_T^\kappa} \in \boldsymbol{\Pi}^0_{4\alpha} ((\Mod^\kappa_T)^2)$;
\item 
if $T$ is not classifiable shallow, then $\cong_T^\kappa$ is not \( \kappa^+ \)-Borel.
\end{itemize}
In particular, $T$ is classifiable shallow if and only if $\mathrm{rk}_B({\cong_T^\kappa}) < \omega_1$.
\end{theorem}

This gives us a non-trivial upper bound for the Borel rank of the isomorphism relation on the \( \kappa \)-sized models of a classifiable shallow theory. We remark once again that it does not depend on the fixed size $\kappa$ of the models. Here we collect a few sample applications of this observation (\( \kappa \) is always assumed to be any cardinal satisfying \( \kappa^{< \kappa} = \kappa > 2^{\aleph_0} \)). The source for these examples (and more!) is \cite{HL97}.
\begin{enumerate-(i)}
\item 
The classifiable theory $T$ of the additive group of integers has depth $1$, thus $\cong_T^\kappa$ is $\boldsymbol{\Pi}^0_{4}$. 
\item 
Fix $\beta<\omega_1$, suppose our language contains binary relation symbols $E_\alpha$ for every $\alpha<\beta$, and define a theory $T^\beta$ such that:
\begin{itemize}
\item[--]
$E_\alpha$ is an equivalence relation for every $\alpha<\beta$;
\item[--] 
if $\gamma<\delta<\beta$, then $E_\gamma$ refines $E_\delta$ and every class in $E_\delta$ contains infinitely many classes of $E_\gamma$;
\item[--]
each class of $E_0$ is infinite.
\end{itemize}
It can be shown that $T^\beta$ is a classifiable shallow theory of depth $\beta+1$, thus $\cong_{T^\beta}^\kappa$ is $\boldsymbol{\Pi}^0_{4\beta+4}$. Furthermore, the disjoint union of theories $T^{\beta_i}$ with $\beta_i$ cofinal in $\gamma$ is a classifiable shallow theory $\bar T^\gamma$ of depth $\gamma$, thus $\cong_{\bar T^\gamma}^\kappa$ is $\boldsymbol{\Pi}^0_{4\gamma}$.
\end{enumerate-(i)}  

In principle,  this approach could be reversed. The above Descriptive Main Gap Theorem~\ref{thm:DMGT} could provide different means to study the stability properties of a theory \( T \). Namely, 
if one succeeds, using descriptive set-theoretical methods,  in proving that \( \cong^\kappa_T \) is \( \kappa^+ \)-Borel for a suitable uncountable \( \kappa \), then we can conclude that \( T \) is classifiable shallow; and if one can also compute \( \mathrm{rk}_B(\cong^\kappa_T ) \), then we have a lower bound for (four times) the depth of $T$. This method could thus turn out to be useful to isolate ``natural'' classifiable shallow theories with higher and higher depth, a notoriously tricky problem. The advantages of this new approach compared to the classical ones would be the following:
\begin{itemize}
\item
There is a lot of freedom in choosing the cardinal \( \kappa \), it is enough that \( \kappa^{< \kappa} = \kappa > 2^{\aleph_0} \).
\item
There is also some freedom in the choice of the set-theoretic universe to work in. For example, any forcing extension of the universe in which all cardinals and the continuum are preserved would be fine.
\item
It could be easier to compute the Borel rank of \( \cong^\kappa_T \) rather than directly computing the depth of \( T \). In particular, we do not need to analyze all models of \( T \) individually, it suffices to look for a ``Borel'' way to classify them up to isomorphism.
\end{itemize}

\section{Further results and open problems}

\subsection{Possible complexities for \( \cong^\kappa_{\ElCl} \)} \label{subsec:possiblecomplexities}

The fact that \( B(\kappa,\ElCl)  = \alpha\) for some ordinal \(  \alpha < \kappa^+ \) tells us that \( {\cong^\kappa_{\ElCl}} \) is either a true  \(  \boldsymbol{\Delta}^0_\alpha \) set, or a  true \( \boldsymbol{\Sigma}^0_\alpha \) set, or  a true \( \boldsymbol{\Pi}^0_\alpha \) set, but it does not distinguish among the three possibilities. Below we provide some additional information on this finer classification of complexities.

The case of \( \boldsymbol{\Sigma}^0_1 \) has been dealt with in Proposition~\ref{prop:open} and Section~\ref{subsec:categoricity}: \( \cong^\kappa_{\ElCl} \) cannot be a true open set, and if \( \ElCl \) is axiomatized by a complete first-order theory \( T \) then \( \cong^\kappa_{\ElCl} \) is (cl)open if and only if \( T \) is \( \kappa \)-categorical. In Example~\ref{xmp:allcomplexities} we have instead seen that for any of the other pointclasses there is an equivalence relation on \( \pre{\kappa}{2} \) lying exactly in that class. 
Quite surprisingly, the next result shows, in particular, that this is no more true if we restrict our attention to isomorphism relations over models of a countable complete first-order theory \( T \) or of an \( \L_{\kappa^+ \kappa} \)-sentence \( \upvarphi \): if \( \kappa^{< \kappa} = \kappa \), then \( \cong^\kappa_T \) and \( \cong^\kappa_\upvarphi \) cannot be a true \( \boldsymbol{\Sigma}^0_\alpha \) set if \( \alpha \) is a limit ordinal.

\begin{theorem} \label{thm:limit}
Let \( \kappa^{< \kappa} = \kappa \), assume that \( |\L| < \kappa \), and let \( \alpha < \kappa^+ \) be a limit ordinal. Then there is no invariant set \( \ElCl \subseteq \Mod^\kappa_\L \) for which \( \cong^\kappa_{\ElCl} \) is a true \( \boldsymbol{\Sigma}^0_\alpha \) set. 
\end{theorem}

\begin{proof}
Suppose that \( {\cong}^\kappa_{\ElCl} \) is \( \boldsymbol{\Sigma}^0_\alpha \). Then \( B(\kappa,\ElCl) = \alpha \), whence \( S(\kappa,\ElCl) = \alpha \) by Corollary~\ref{cor:orig3}\ref{cor:orig3-2}, so that  \( {\cong}^\kappa_\ElCl \) is \( \boldsymbol{\Pi}^0_\alpha \) by Theorem~\ref{orig3}\ref{orig3-2}. This shows that \( {\cong}^\kappa_\ElCl \) is \( \boldsymbol{\Delta}^0_\alpha \), i.e.\ it is not a true \( \boldsymbol{\Sigma}^0_\alpha \) set. 
\end{proof}

Notice that Theorem~\ref{thm:limit} applies to \( \kappa = \omega \) as well: to the best of our knowledge, this is a new observation also in this context. In contrast, we will see in Proposition~\ref{prop:xmp} that, working in \( \mathsf{ZFC} \) alone, \( \cong^\kappa_{\ElCl} \) may be a true \( \boldsymbol{\Pi}^0_\alpha \) set or a true \( \boldsymbol{\Delta}^0_\alpha \) set for appropriate \(\alpha\)'s, even when restricting the attention to first-order elementary classes \( \ElCl = \Mod^\kappa_T \).

The next result provides other nontrivial limitations to the possible complexities of \( \cong^\kappa_{\ElCl} \) (for some specific cardinals \( \kappa \)) when \( \ElCl \) is first-order axiomatizable. 

\begin{prop} \label{prop:GCH}
Assume that \( \kappa = \aleph_\gamma \) is such that \( \kappa^{< \kappa} = \kappa \) and \( \beth_{\omega_1} \left( |\gamma| \right) \leq \kappa \). Then there is no countable complete first-order theory \( T \) such that \( {\cong^\kappa_T } \) is a true \( \boldsymbol{\Sigma}^0_\alpha \) or a true \( \boldsymbol{\Pi}^0_\alpha \) for any \( 1 \leq \alpha < \kappa^+ \).
\end{prop}

\begin{proof}	
Let \( T \) be any countable complete first-order theory, and assume that \linebreak\mbox{\( \mathrm{rk}_B({\cong}^\kappa_T) = \alpha \)} for some \( 1 \leq \alpha < \kappa^+ \). Then for all \( \M \in \Mod^\kappa_T \) 
we also have \( \mathrm{rk}_B([\M]_{\cong}) = \alpha \) by Fact~\ref{fct:fromertoclass}. By 
Remark~\ref{rmk:She}  there are \( <\kappa \)-many \mbox{\( \cong^\kappa_T \)-equivalence} classes, hence the complement of any $\cong^\kappa_T$-equivalence class is a union of $<\kappa$-many of them.  
Since
\( \kappa^{< \kappa} = \kappa \) implies that \( \boldsymbol{\Pi}^0_\alpha(\Mod^\kappa_T) \) is 
closed under unions of size \( < \kappa \) (and the same trivially holds for 
\( \boldsymbol{\Sigma}^0_\alpha(\Mod^\kappa_T) \)), it follows that 
\mbox{\( [\M]_{\cong} \in \boldsymbol{\Delta}^0_\alpha(\Mod^\kappa_T) \)} for every 
\mbox{\( \M \in \Mod^\kappa_T \).} By Fact~\ref{fct:fromclasstoer} and the fact that there are \( < \kappa \)-many \mbox{\( \cong^\kappa_T \)-equivalence} classes, we then conclude that both \( \cong^\kappa_T \) and \( (\Mod^\kappa_T)^2 \setminus {\cong^\kappa_T} \) are \( \boldsymbol{\Sigma}^0_\alpha \), whence \( \cong^\kappa_T \) is \( \boldsymbol{\Delta}^0_\alpha \).
\end{proof}

\begin{remark} \label{rmk:GCH}
As argued in Remark~\ref{rmk:She} it is not difficult to find cardinals satisfying the hypothesis of Proposition~\ref{prop:GCH}. For example, under \( \mathsf{GCH} \) it is enough to pick any \( \omega_1 \leq \delta, \gamma \in \On \) with \( |\gamma| \geq |\delta| \) and \(\delta\) a \emph{successor} ordinal, and then set \( \kappa = \aleph_{\gamma+\delta} \) (the requirement that \( \delta \) be successor is to ensure that \( \kappa^{< \kappa} = \kappa \)).
\end{remark}

Motivated by the above partial results, we end this section with the following very general question (compare it with Proposition~\ref{prop:xmp} below).

\begin{question}
For which infinite cardinals \( \kappa \) and classes \( \boldsymbol{\Gamma} \in \{ \boldsymbol{\Sigma}^0_\alpha,\boldsymbol{\Pi}^0_\alpha , \boldsymbol{\Delta}^0_\alpha \} \) with \( 1 \leq \alpha < \kappa^+  \)  there is an invariant set \( \ElCl \subseteq \Mod^\kappa_\L \) such that \( \cong^\kappa_{\ElCl} \) is a true \( \boldsymbol{\Gamma} \) set? In particular, is there any \( \ElCl \) and \( \kappa \) as above such that \( \cong^\kappa_\ElCl \) is a true \( \boldsymbol{\Sigma}^0_\alpha \) set for some \( 1 \leq \alpha < \kappa^+ \), at least consistently? If yes, can \( \ElCl \) be taken to be a first-order elementary class or an \( \L_{\kappa^+ \kappa} \)-elementary class?
\end{question}

\subsection{An example} \label{subsec:someexamples}

In this section we show that the values of \( B(\kappa,\ElCl) \) and \( S(\kappa,\ElCl) \) may depend on the cardinal \( \kappa \), and that they may differ from each other when they are successor ordinals (compare this with Corollary~\ref{cor:orig3}\ref{cor:orig3-2}). In particular, this can happen even when restricting to invariant sets of the form \( \Mod^\kappa_T \) for \( T \) a countable complete first-order theory in a finite language.

Let \( \L = \{ P \} \) be the language consisting of just one relational symbol, and let \( T \) be the countable complete first-order \( \L \)-theory asserting that there are infinitely many elements which satisfy \( P \) and infinitely many elements which do not. The isomorphism type of a model \( \M \) of \( T \) is uniquely determined by the cardinality of \( P^\M \) and of its complement. In particular, the cardinality of (at least) one of these two sets must equal the size of \( \M \), while the other set may have any intermediate infinite cardinality. Thus if we consider models of size \( \kappa = \aleph_\alpha \) then there are \mbox{\( | \alpha | \)-many} isomorphism types if \( \alpha \geq \omega \), and \( 2n+1 \)-many ones if \( \alpha = n < \omega \); in particular, there are always \( \leq \kappa \)-many of them, and if \( \kappa \) is not a fixed point of the \( \aleph \) function, then there are \( < \kappa \)-many ones. Notice also that \( T \) is \( \aleph_0 \)-categorical but not uncountably categorical.

\begin{prop} \label{prop:xmp}
Let \( \kappa^{< \kappa} = \kappa > \omega \). 
\begin{enumerate-(1)}
\item \label{prop:xmp-1}
If \( \kappa  = \lambda^+ \) is a successor cardinal, then \( S(\kappa,T) = 2 \) and \( \cong^\kappa_T \) is a true \( \boldsymbol{\Delta}^0_3 \) set. In particular, \( B(\kappa,T) = 3 \).
\item \label{prop:xmp-2}
If \( \kappa \) is a limit cardinal, then \( S(\kappa,T) = 1 \) while \( \cong^\kappa_T \) is a true \( \boldsymbol{\Pi}^0_2\) set. In particular, \( B(\kappa,T) = 2 \).
\end{enumerate-(1)}
\end{prop}

\begin{proof}
We first consider the case \( \kappa = \lambda^+ \), and begin with the computation of \( S(\kappa,T) \).
\begin{claim} \label{successorS>1}
Let \( \M,\N \in \Mod^\kappa_T \) be such that \( |P^\M| = \lambda\) while \( |\kappa \setminus P^\M| = | P^\N | = | \kappa \setminus P^\N| = \kappa = \lambda^+ \). Then \( \2 \wins  \EF_1^\kappa(\M,\N) \).
\end{claim}

\begin{proof}[Proof of the claim]
Any run of \( \EF_1^\kappa(\M,\N) \) consists of just one round where \( \1 \) 
provides a set \( C \subseteq \kappa \) of size \( \leq \lambda \) and \( \2 \) has to respond with a partial isomorphism \( f \) 
between \( \M \) and \( N \) of size \( < \kappa \) and  such that \mbox{\( C \subseteq \dom(f) \cap \ran(f) \). }
But clearly this is  always possible: just let \( \dom(f) = \ran(f) = D \supseteq C \) be any subset of 
\( \kappa \) such that all of \( D \cap P^\M \), \( D \setminus P^\M \), \( D \cap P^\N \), and \( D \setminus P^\N \) have 
size \(\lambda\), and then define \( f \) in the obvious way. 
\end{proof}

\begin{claim} \label{successorS=2} 
If \( \M,\N \in \Mod^\kappa_T \) are not isomorphic, then 
\( \1 \wins \EF_2^\kappa(\M,\N) \) (whence \( \2 \not\wins \EF_2^\kappa(\M,\N) \)). 
\end{claim}

\begin{proof}[Proof of the claim]
Since 
\( \M \not \cong \N \)
we have that either \( |P^\N| \neq |P^\M| \)  or \( |\kappa \setminus P^\M| \neq |\kappa \setminus P^\N | \).
Without loss of generality,  
we may assume 
\( | P^\M | < |P^\N| \leq \kappa \) (the other cases are similar). 
Let \( \1 \) play \( P^\M \) in the first round of 
the game,
and let \( f \) 
be the move of \( \2 \) in this first round, which may be assumed to be a partial isomorphism (otherwise \( \1 \) already won). 
Notice that it cannot happen that \( \ran(f) \supseteq P^\N \) because of the cardinality assumption on \( P^\M \) and \( P^\N \). Thus on the second round
\( \1 \) can play any \( D \supseteq P^\M \)
of size 
\( < \kappa \) containing at least one point in \( P^\N \setminus \ran(f) \),
and \( \2 \) will not be able to extend 
\( f \) to a partial isomorphism with range extending \( D \) because there are no more points in 
\( (\kappa \setminus \dom(f)) \cap P^\M  \). 
\end{proof}

On the one hand \( S(\kappa,T) > 1 \) because  the structures \( \M \) and \( \N \) considered in Claim~\ref{successorS>1} are not isomorphic. On the other hand, Claim~\ref{successorS=2} yields \( S(\kappa,T) \leq 2 \) by contrapositive. Thus \( S(\kappa,T) = 2 \).

\smallskip

Now we compute the topological complexity of \( \cong^\kappa_T \). Let \( \hat{\M} \in \Mod^\kappa_T\) be such that \( |P^{\hat{\M}}| = | \kappa \setminus P^{\hat{\M}}| = \kappa \).

\begin{claim} \label{claimPi02class}
 \( [\hat{\M}]_{\cong} \) is a true \( \boldsymbol{\Pi}^0_2(\Mod^\kappa_T) \) set. 
 \end{claim} 

\begin{proof}[Proof of the claim] 
 Indeed, 
\[ 
[\hat\M]_{\cong} = \bigcap_{\alpha < \kappa} \bigcup_{\alpha \leq \beta,\beta' < \kappa} \{ \N \in \Mod^\kappa_T \mid \beta \in P^\N \wedge \beta' \notin P^\N \},
 \] 
whence \( [\hat\M]_{\cong} \in \boldsymbol{\Pi}^0_2(\Mod^\kappa_T) \). On the other hand, the function \( f \)
sending \( x \in \pre{\kappa}{2} \) to the structure \( \N \in \Mod^\kappa_T \) such that
\[ 
P^\N = \lambda \cup \{ \lambda+2\alpha \mid x(\alpha) = 1 \}
 \] 
is continuous and such that \( f^{-1}([\hat\M]_{\cong}) = P \) where
\[ 
P = \{ x \in \pre{\kappa}{2} \mid \forall \alpha < \kappa \, \exists \alpha \leq \beta < \kappa \, (x(\beta) = 1)	 \}.
 \] 
Since the latter is a well-known true \( \boldsymbol{\Pi}^0_2 (\pre{\kappa}{2}) \) set,  \( [\hat\M]_{\cong} \) cannot be \( \boldsymbol{\Sigma}^0_2(\Mod^\kappa_T) \) and we are done. 
\end{proof}

\begin{claim} \label{claimsuccessorotherclasses}
If \( \N \in \Mod^\kappa_T \setminus [\hat{\M}]_{\cong} \), then  \( [\N]_{\cong} \in \boldsymbol{\Sigma}^0_2 (\Mod^\kappa_T) \). 
\end{claim}

\begin{proof}[Proof of the claim]
Let us consider the case where \( \N \) is such that \( |P^\N| = \lambda \) (the other cases are similar). We have
\begin{align*}
[\N]_{\cong} = \bigcup \Big\{ A_\Q \cap \Mod^\kappa_T\mid {} & \Q \text{ is a \(\lambda\)-sized structure} \\
&\text{with domain } \subseteq \kappa \text{ and } |P^\Q| = \lambda \Big\},
 \end{align*}
where
\[ 
A_\Q = \Nbhd_\Q \cap  \{ \R \in \Mod^\kappa_\L \mid P^\R = P^\Q \} . 
 \] 
Since \( \{ \R \in \Mod^\kappa_\L \mid P^\R = P^\Q \} \) is closed in \( \Mod^\kappa_\L \), then so is \( A_\Q \), whence \( [\N]_\Q \) is a union of \( \kappa \)-many closed sets by \( \kappa^{< \kappa} = \kappa \). 
\end{proof}

\begin{claim} \label{claimSigma02class}
Let \( \hat{\N} \in \Mod^\kappa_T \) be such that \( |P^{\hat{\N}}| = \lambda \). Then \( [\hat{\N}]_{\cong} \) is a true \( \boldsymbol{\Sigma}^0_2(\Mod^\kappa_T) \) set.
\end{claim}

\begin{proof}[Proof of the claim]
By Claim~\ref{claimsuccessorotherclasses} it is enough to show that \( [\hat{\N}]_{\cong} \notin \boldsymbol{\Pi}^0_2(\Mod^\kappa_T) \). But if \( f \) and \( P \) are as in the proof of Claim~\ref{claimPi02class}, then \( f^{-1}([\hat{\N}]_{\cong}) = \pre{\kappa}{2} \setminus P \): since \( \pre{\kappa}{2} \setminus P \) is a true \( \boldsymbol{\Sigma}^0_2(\pre{\kappa}{2}) \) set, we are done.
\end{proof}

Since there are at most \( \kappa \)-many isomorphism types for \( \kappa \)-sized models of \( T \), by Claim~\ref{claimPi02class} and Claim~\ref{claimsuccessorotherclasses} we get that both
\[ 
{\cong^\kappa_T} = \bigcup \{ [\M]_{\cong} \times [\M]_{\cong} \mid \M \in \Mod^\kappa_T \}
 \] 
and
\[ 
(\Mod^\kappa_T)^2 \setminus {\cong^\kappa_T} = \bigcup \{ [\M]_{\cong} \times [\N]_{\cong} \mid \M, \N \in \Mod^\kappa_T \text{ and } \M \not\cong \N \}
 \] 
belong to \( \boldsymbol{\Sigma}^0_3 \), whence \( {\cong^\kappa_T} \in \boldsymbol{\Delta}^0_3 \). Finally, \( {\cong^\kappa_T} \notin \boldsymbol{\Sigma}^0_2 \) and \( {\cong^\kappa_T} \notin \boldsymbol{\Pi}^0_2 \) by Claim~\ref{claimPi02class} and Claim~\ref{claimSigma02class}, respectively, together with Fact~\ref{fct:fromertoclass}.

\medskip

We now consider a limit cardinal \( \kappa \), and again compute first \( S(\kappa,T) \). 

\begin{claim}
If \( \M,\N \in \Mod^\kappa_T \) are not isomorphic, then 
\( \1 \wins \EF_1^\kappa(\M,\N) \) (whence \( \2 \not\wins \EF_1^\kappa(\M,\N) \)).
\end{claim}

\begin{proof}[Proof of the claim]
Since \( \M \not\cong \N \) we have that either \( |P^\M| \neq |P^\N| \) or \linebreak\mbox{\( | \kappa \setminus P^\M| \neq |\kappa \setminus P^\N| \).} Without loss of generality, we may assume \( |P^\M| < |P^\N| \) (the other cases are similar). Since \( \kappa \) is limit, there is a cardinal \( \lambda < \kappa \) such that \( |P^\M| < \lambda \leq |P^\N| \). So \( \1 \) can play any subset \( C \) of \( P^\N \) of size \(\lambda\) as her first (and unique) move, and by choice of \(\lambda\) player \( \2 \) will not be able to produce a partial isomorphism between \( \M \) and \( \N \) with range containing \( C \).
\end{proof}

By contrapositive, \( S(\kappa,T) \leq 1 \). On the other hand,  \( S(\kappa,T) > 0 \) because \( T \) is not uncountably categorical. Thus \( S(\kappa,T) = 1 \).

\smallskip

To compute the topological complexity of \( \cong^\kappa_T \), notice that for every \mbox{\( \M ,\N \in \Mod^\kappa_T \)} one has \( \M \cong \N \) if and only if
\begin{quotation}
for all cardinals \( \lambda < \kappa \), there are at least \(\lambda\)-many elements in \( P^\M \) if and only if  there are at least \(\lambda\)-many elements in \( P^\N \), and the same when replacing \( P^\M \) and \( P^\N \) with \( \kappa \setminus P^\M \) and \( \kappa \setminus P^\N \), respectively.
\end{quotation} 
(Here it is crucial that \( \kappa \) is  a limit cardinal to ensure that if for all \(\lambda < \kappa \) there are at least \(\lambda\)-many elements in \( P^\M \), then \( |P^\M| = \kappa \), and similarly for  \( P^\N \), \( \kappa \setminus P^\M \), and \( \kappa \setminus P^\N \).)  The above condition easily yields that \( {\cong^\kappa_T} \in \boldsymbol{\Pi}^0_2 \). To see that \( {\cong^\kappa_T} \) does not belong to any lower class, just observe that Claim~\ref{claimPi02class} holds for limit \( \kappa \)'s as well and use again Fact~\ref{fct:fromertoclass}. 
\end{proof}

Notice that in part~\ref{prop:xmp-2} the relation \( \cong^\kappa_T \) has the maximal complexity allowed by Theorem~\ref{link2}. We also remark that Proposition~\ref{prop:xmp}\ref{prop:xmp-2} does not contradict Proposition~\ref{prop:GCH} because it deals with regular limit (i.e.\ weakly inaccessible) cardinals, which in models of \( \mathsf{GCH} \) are inaccessible and thus limit points of the \( \aleph \)-function --- by Remark~\ref{rmk:She} in such a situation the upper bound on the number of models given by Theorem~\ref{thm:She} is trivial and the proof of Proposition~\ref{prop:GCH} does not go through.

\subsection{Borel reducibility} \label{subsec:Borelreducibility}

Borelness is a very strong dividing line among the possible complexities of isomorphism relations of the form \( \cong^\kappa_T \): knowing that \( \cong^\kappa_T \) is \( \kappa^+ \)-Borel means that there is a procedure involving only \( \kappa \)-ary Boolean operations and with a fixed length \( \alpha < \kappa^+ \) which allows us to classify the \( \kappa \)-sized model of \( T \) up to isomorphism, while if \( \cong^\kappa_T \) is not \( \kappa^+ \)-Borel then there is no such algorithm. A finer complexity analysis is provided by \( \kappa^+ \)-Borel reducibility.

\begin{defin}
Let \( X,Y \) be topological spaces homeomorphic to a \( \kappa^+ \)-Borel subset of \( \pre{\kappa}{2} \), and let \( E,F \) be binary relations on \( X,Y \), respectively. A function \( f \colon X \to Y \) is called a \emph{reduction} of \( E \) to \( F \) is for all \( x,x' \in X \)\
\[ 
x \mathrel{E} x' \IFF f(x) \mathrel{F} f(x').
 \] 
We say that \( E \) is \emph{\( \kappa^+ \)-Borel reducible} to \( F \), in symbols \( E \leq^\kappa_B F \), if there is a \mbox{\( \kappa^+ \)-Borel} measurable function which is a reduction of \( E \) to \( F \). We also set \( E <^\kappa_B F \) if \( E \leq^\kappa_B F \) but \( F \not\leq^\kappa_B E \).
\end{defin}

A possible interpretation of the statement ``\( E <^\kappa_B F \)'' is that \( F \) is strictly more complicated than \( E \).
Obviously, if \( E \leq^\kappa_B F \) and \( F \) is \( \kappa^+ \)-Borel (respectively, \( E \) is not \( \kappa^+ \)-Borel), then \( E \) is \( \kappa^+ \)-Borel (respectively, \( F \) is not \( \kappa^+ \)-Borel). However, it is not true that  all \( \kappa^+ \)-Borel equivalence relations are \( \kappa^+ \)-Borel reducible to each other, and the fact that \( E \) is \( \kappa^+ \)-Borel and \( F \) is not does not in general  imply that \( E \leq^\kappa_B F \). Thus \( \kappa^+ \)-Borel reducibility provides a complexity analysis which is  strictly finer than the distinction \( \kappa^+ \)-Borel versus non-\( \kappa^+ \)-Borel.

The following easy observation shows that, for suitable cardinals \( \kappa \), there is also a gap in the sense of \( \kappa^+ \)-Borel reducibility between classifiable shallow theories and those  which are not. 
(Notice that it makes sense to consider \( \kappa^+ \)-Borel reducibility among isomorphism relations of the form \( \cong^\kappa_T \) because \( \Mod^\kappa_T \) is a \( \kappa^+ \)-Borel subset of \( \Mod^\kappa_\L \) by Theorem~\ref{lopezext}, and  the latter is isomorphic to \( \pre{\kappa}{2} \).)

\begin{prop} \label{prop:reducibility} 
Let \( \kappa = \aleph_\gamma \) be such that \( \kappa^{< \kappa} = \kappa \) and \( \beth_{\omega_1} \left( |\gamma| \right) \leq \kappa \). Let \( T ,T' \) be arbitrary countable complete first-order theory, and assume that  \( T \) is classifiable shallow, while \( T' \) is not. Then 
\[ 
{\cong^\kappa_T} <^\kappa_B {\cong^\kappa_{T'}}.
 \] 
Moreover, the \( \leq^\kappa_B \)-interval between the two isomorphism relations is nonempty, that is, there is a (\( \kappa^+ \)-Borel) equivalence relation \( E \) such that
\[ 
{\cong^\kappa_T} <^\kappa_B E <^\kappa_B {\cong^\kappa_{T'}}.
 \] 
\end{prop}

\begin{proof}
Let \(\lambda\) be the number of \( \cong^\kappa_T \)-equivalence classes: by Theorem~\ref{thm:She} and Remark~\ref{rmk:She} we have \( \lambda < \kappa \). Let \( E \) be any \( \kappa^+ \)-Borel equivalence relation on a topological space \( X \) (homeomorphic to a \( \kappa^+ \)-Borel subset of \( \pre{\kappa}{2} \)) having exactly \( \kappa \)-many classes. Pick representatives \( \M_i \), \( i < \lambda \), in each \( \cong^\kappa_T \)-equivalence class, and pairwise \( E \)-inequivalent elements \( (x_i)_{i < \lambda} \). Then the map \( f \colon \Mod^\kappa_T \to X \) sending \( \N \in\Mod^\kappa_T \) to the unique \( x_i \) for which \( \N \cong \M_i \) is trivially a reduction of \( \cong^\kappa_T \) to \( E \), and it is \( \kappa^+ \)-Borel because the \( f \)-preimage of any subset of \( X \) is a union of \( \leq \lambda \)-many \( {\cong^\kappa_T} \)-equivalence classes, which are all \( \kappa^+ \)-Borel by Fact~\ref{fct:fromertoclass} and the fact that \( \cong^\kappa_T \) is \( \kappa^+ \)-Borel by Theorem~\ref{thm:FHK}. Moreover \( E \not\leq^\kappa_B {\cong^\kappa_T} \) because of the number of equivalence classes, hence \( {\cong^\kappa_T} <^\kappa_B E \).

To prove \( E <^\kappa_B {\cong^\kappa_{T'}} \) we use the same idea. Indeed, \( {\cong^\kappa_{T'}} \not\leq^\kappa_B E \) because by Theorem~\ref{thm:She} there are \( 2^\kappa \)-many \( \cong^\kappa_{T'} \)-equivalence classes, while there are only \( \kappa \)-many \( E \)-equivalence classes. To produce a \( \kappa^+ \)-Borel reduction of \( E \) to \( \cong^\kappa_T \) pick instead representatives \( x_i \), \( i < \kappa \), in each \( E \)-equivalence class, and pairwise non-isomorphic models \( (\N_i)_{i < \kappa} \) in \( \Mod^\kappa_{T'} \): the map \( g \colon X \to \Mod^\kappa_{T'} \) sending each \( y \in X \) to the unique \( \N_i \) such that \( y \mathrel{E} x_i \) is clearly a reduction, and it is \( \kappa^+ \)-Borel because \( E \) is.
\end{proof}

Notice that the conditions on \( \kappa \) in Proposition~\ref{prop:reducibility} are the same of Proposition~\ref{prop:GCH}, thus Remark~\ref{rmk:GCH} applies here as well. Moreover, the proof above actually shows that under the same assumptions on \( \kappa \), if \( T, T' \) are both classifiable shallow theories then \( {\cong^\kappa_T} \leq^\kappa_B {\cong^\kappa_{T'}} \) if and only if \( I(\kappa,T) \leq I(\kappa,T') \), thus the isomorphism relations over \( \kappa \)-sized models of classifiable shallow theories are prewellordered under \mbox{\( \kappa^+ \)-Borel} reducibility. Other variations along the same lines are of course possible; we leave them to the interested reader.

In the same vein, Hyttinen, Kulikov, and Moreno considered in~\cite{HKM17} another dividing line among countable complete first order theories, and proved the following descriptive set-theoretic gap.

\begin{theorem}[{\cite[Theorem 6]{HKM17}}] \label{thm:HKM17}
Assume \( \mathsf{V = L} \). Let \( \kappa = \kappa^{< \kappa} = \lambda^+ \) with \( 2^\lambda > 2^{\aleph_0} \) and \( \lambda^{< \lambda} = \lambda \). Let \( T \), \( T' \) be arbitrary countable complete first-order theories, and assume that  \( T \) is classifiable, while \( T' \) is not. Then 
\[ 
{\cong^\kappa_T} <^\kappa_B {\cong^\kappa_{T'}}.
 \] 
The same conclusion can be forced to hold over any model of \( \mathsf{ZFC} \) through a \( \kappa \)-closed \( \kappa^+ \)-cc forcing notion.
\end{theorem}

Comparing Proposition~\ref{prop:reducibility} with Theorem~\ref{thm:HKM17} one may notice the following:
\begin{itemizenew}
\item
The dividing line in Theorem~\ref{thm:HKM17} is classifiability, while the dividing line in Proposition~\ref{prop:reducibility} coincides with the one of Shelah's Main Gap Theorem~\ref{thm:She}.
\item
Theorem~\ref{thm:HKM17} is a consistency result which holds in certain specific models of \( \mathsf{ZFC} \), namely G\"odel's constructible universe \( \mathsf{L} \) or certain forcing extensions of \( \mathsf{V} \). It is apparently open whether one can get such a result in \( \mathsf{ZFC} \) alone (and possibly with less constraints on \( \kappa \)) --- see the Question at the end of~\cite{HKM17}. In contrast, Proposition~\ref{prop:reducibility} is proved in \( \mathsf{ZFC} \) alone.
\item
The conditions on \( \kappa \) in the two results are quite different. If e.g.\ we work in \( \mathsf{L} \), then the successors of inaccessible cardinals satisfy the hypotheses of Theorem~\ref{thm:HKM17} but not those of Proposition~\ref{prop:reducibility}; conversely, there are successors of singular cardinals (of any cofinality) which satisfy the hypotheses of Proposition~\ref{prop:reducibility} but not those of Theorem~\ref{thm:HKM17}.
\item
Proposition~\ref{prop:reducibility} is just an easy observation following mostly from cardinality considerations together with Theorem~\ref{thm:FHK}. Theorem~\ref{thm:HKM17} is instead much more informative and requires involved techniques, and it is arguably stronger in at least two different directions: it can be shown that the \( \kappa^+ \)-Borel reduction between \( \cong^\kappa_T \) and \( \cong^\kappa_{T'} \) can actually be taken to be continuous; moreover, the \( \leq^\kappa_B \)-gap between the two isomorphism relations can be shown to be very large and complicated (see~\cite[Theorem 7]{HKM17}).
\end{itemizenew}

The last item naturally raises the following question.

\begin{question}
Let \( \kappa \) be as in Proposition~\ref{prop:reducibility}. How large can be the \( \leq^\kappa_B \)-gap between \( \cong^\kappa_T \) and \( \cong^\kappa_{T'} \) when \( T \) is classifiable shallow and \( T' \) is not? In particular, what happens if \( T' \) is classifiable deep? (By Theorem~\ref{thm:HKM17} this is the unique relevant case that needs to be studied, if we work in \( \mathsf{L} \) and further assume that \( \kappa \) satisfies the hypotheses of that theorem.)
\end{question}

\subsection{Incomplete theories} \label{subsec:incompletetheories}

Expanding on a suggestion of M.\ Moreno, we notice that weaker forms of the Descriptive Main Gap Theorem~\ref{thm:main} apply to more general situations, including the case of countable theories \( T \) which are not necessarily complete (simply notice that \( \Mod^\kappa_T \) is trivially closed under elementary equivalence).  We denote by \( \mathrm{Th}(\M) \) the (complete) theory of a structure \( \M \).

\begin{theorem} \label{thm:incompleteDMG}
Let $\kappa^{<\kappa}=\kappa>2^{\aleph_0}$. Let \( \L \) be a countable first-order language, and \( \ElCl \subseteq \Mod^\kappa_\L \) be any class 
closed under elementary equivalence.
Then either \( B(\kappa,\ElCl) = \mathrm{rk}_B({\cong_\ElCl^\kappa}) \leq \omega_1 \), or else \( \cong^\kappa_\ElCl \) is not \( \kappa^+ \)-Borel at all.
\end{theorem}

\begin{proof}
Let \( (T_\alpha)_{\alpha < \nu } \) (for some \( \nu \leq 2^{\aleph_0} \)) be an enumeration without  repetitions of the complete theories of the form \( \mathrm{Th}(\M) \) for \( \M \in \ElCl \), so that \( \Mod^\kappa_{T_\alpha} \subseteq \ElCl \) for every \( \alpha < \nu \) by the hypothesis on \( \ElCl \). Notice that, by the choice of the \( T_\alpha \)'s, for every \( \M \in \ElCl \) there is a unique \( \alpha < 2^{\aleph_0} \) such that \( \M \in \Mod^\kappa_{T_\alpha} \), namely the \( \alpha < \nu \) such that \( T_\alpha = \mathrm{Th}(\M) \). We distinguish two cases.

If for some \( \alpha < 2^{\aleph_0} \) the theory \( T_\alpha \) is not classifiable shallow, then \( \cong^\kappa_{T_\alpha} \) is not \( \kappa^+ \)-Borel by Theorem~\ref{thm:FHK}, and hence the same applies to the whole \( \cong^\kappa_\ElCl \).

Assume now that all \( T_\alpha \)'s are classifiable shallow, so that \( \mathrm{rk}_B({\cong^\kappa_{T_\alpha}}) < \omega_1 \) by Theorem~\ref{thm:DMGT}, and let \( B_\alpha \subseteq (\Mod^\kappa_\L)^2 \) be corresponding \( \kappa^+ \)-Borel sets of rank \( < \omega_1 \) such that \( B_\alpha \cap (\Mod^\kappa_{T_\alpha})^2 = {\cong^\kappa_{T_\alpha}} \). Then for all \( \M,\N \in \ElCl \)
\[ 
\M \cong \N \IFF \bigvee_{\alpha < \nu} \left( {\M \in \Mod^\kappa_{T_\alpha}} \wedge {\N \in \Mod^\kappa_{T_\alpha}} \wedge {(\M,\N) \in B_\alpha} \right).
 \] 
Since by Remark~\ref{rmk:borelrankfirstordertheories} all the \( \Mod^\kappa_{T_\alpha} \) are \( \kappa^+ \)-Borel subsets of  \( \Mod^\kappa_\L \)  with rank \( \leq \omega \), the formula in parentheses defines \( \kappa^+ \)-Borel sets with rank%
\footnote{Notice however that when \( \alpha \) varies in \( 2^{\aleph_0} \) such ranks may be arbitrarily high below \( \omega_1 \).} 
\( < \omega_1 \), hence they are all in \( \boldsymbol{\Delta}^0_{\omega_1} \). Since such class is closed under \( < \kappa \)-unions by \( \kappa^{< \kappa} = \kappa \) and we assumed \( 2^{\aleph_0} < \kappa \), it follows that \( {\cong^\kappa_\ElCl} \in \boldsymbol{\Delta}^0_{\omega_1}(\ElCl^2) \).
\end{proof}

Of course a similar argument can be used to extend Shelah's Main Gap Theorem~\ref{thm:She} to classes \( \ElCl \) as in Theorem~\ref{thm:incompleteDMG}. In particular, if \( \kappa \geq \aleph_1 \) is the \( \gamma \)-th cardinal and \( T \) is a countable (not necessarily complete) first-order theory, then either \( I(\kappa, T) \leq \beth_{\omega_1} \left(|\gamma|\right) \), or else \( I(\kappa,T) = 2^\kappa \).

\end{document}